\documentclass[10pt,twoside,reqno,centertags,draft]{amsart}
\usepackage{amsfonts}
\usepackage{color,enumitem,graphicx}
\usepackage[colorlinks=true,urlcolor=blue,
citecolor=red,linkcolor=blue,linktocpage,pdfpagelabels,
bookmarksnumbered,bookmarksopen]{hyperref}
  \usepackage{amsmath,amsthm,amsfonts,amssymb}

\newcommand{\N}{\mathbb{N}}
\newcommand{\R}{\mathbb{R}}

\newcommand{\cA}{{\mathcal A}}

\newcommand{\cE}{{\mathcal E}}
\newcommand{\cF}{{\mathcal F}}

\newcommand{\norm}[1]{\lVert#1\rVert}

\newcommand{\s}{s}

\newcommand{\equ}[1]{(\ref{#1})}

\numberwithin{equation}{section}

\newtheorem{lemma}{Lemma}[section]
\newtheorem{theorem}{Theorem}[section]
\newtheorem{definition}{Definition}[section]
\newtheorem{proposition}{Proposition}[section]
\newtheorem{remark}{Remark}[section]
\newtheorem{corollary}{Corollary}[section]

\begin{document}
\title{ Existence and nonexistence of solutions for fractional elliptic equations arising from closed MEMS model    }
\date{}
\maketitle

\vspace{ -1\baselineskip}

\begin{center}

  {\small  Huyuan Chen$^{*,}$\footnote{chenhuyuan@yeah.net}\qquad  
  Jialei Jiang$^{\dag,}$\footnote{Jiangjialei2023@126.com} \qquad  
  Jun Wang$^{\ddag,}$\footnote{wangmath2011@126.com}}
 \medskip

 {\small   $^*$\,Center for Mathematics and Interdisciplinary Sciences, Fudan University, \\
Shanghai 200433,  PR China \\
and\\
 Shanghai Institute for Mathematics and Interdisciplinary Sciences (SIMIS), \\
Shanghai 200433,  PR China} \\[3mm]

 {\small   $^\dag$\,Department of Mathematics, Jiangxi Normal University, Nanchang,\\ Jiangxi 330022, PR China   } \\[3mm]

 {\small    $^\ddag$\,School of Mathematical Sciences, Jiangsu University, Zhenjiang,\\ Jiangsu  212013, PR China   } \\[3mm]
 \medskip

\end{center}

\begin{quote}
%\footnotesize
{\bf Abstract.} The objective of our paper is to investigate
fractional elliptic equations of the form $(-\Delta)^s
u=\frac{\lambda }{(a-u)^2}$ within a bounded domain $\Omega$,
subject to zero Dirichlet boundary conditions. Here, $s\in(0,1)$,
$\lambda>0$, and the function $a$ vanishes at the boundary while
satisfying additional conditions. This problem originates from
Micro-Electromechanical Systems (MEMS) devices, particularly when the
elastic membrane makes contact with the ground plate at the
boundary. We establish both existence and nonexistence results,
illustrating how the boundary decay of the membrane influences the
solutions and pull-in voltage.
\end{quote}

   \smallskip

\noindent {\qquad \quad \ \small {\bf Keywords}:  MEMS,  Fractional Laplacian, Existence, Pull-in Voltage.}

   \smallskip

 \noindent {\qquad \quad \ \small {\bf Mathematics Subject Classification}:  36J08, 35B50, 35J15.}

  \medskip

% }

\setcounter{equation}{0}
\section{Introduction}

Let $\Omega$ be a $C^2$ connected bounded domain in $\R^N$ with $N\ge 2$,
the function  $a:\bar\Omega\to[0,1]$ be in $C^{0,\beta}_{\rm loc}(\Omega)\cap C(\bar\Omega)$ with $\beta\in(0,1)$.
Our purpose of this paper is to study the existence and nonexistence of minimal solutions to the
   fractional elliptic equation
\begin{equation}\label{eq 1.1}
\left\{\arraycolsep=1pt
\begin{array}{lll}
 (-\Delta)^s    u = \frac{\lambda }{(a-u)^2}\quad  &{\rm in}\quad\ \Omega,
 \\[2mm]
 \phantom{-- }
 0<u\leq a\quad &{\rm in}\quad\ \Omega,
 \\[2mm]
 \phantom{--  }
 \quad\ u=0\quad &{\rm in}\quad\   \mathbb{R}^N\backslash \Omega,
\end{array}
\right.
\end{equation}
where  $\lambda>0$ and  $(-\Delta)^\s$ with $\s\in(0,1)$  is the fractional Laplacian defined as
 \begin{equation}\label{fl 1}
 (-\Delta)^\s  u(x)=\frac{c_{N,\s}}{2} P.V. \int_{\R^N}\frac{2u(x)-u(x+y)-u(x-y)}{|y|^{N+2\s}}  dy,
\end{equation}
here   $c_{N,\s}$ is a normalized constant as
$$c_{N,\s}=2^{2\s}\pi^{-\frac N2}\s\frac{\Gamma(\frac{N+2\s}2)}{\Gamma(1-\s)},$$
P.V. denotes the principal value of the integral, for notational simplicity we omit in what follows.

The equation denoted by (\ref{eq 1.1}) originates from scenarios
involving Micro-Electromechanical Systems (MEMS) devices,
particularly when the elastic membrane makes contact with the ground
plate along the boundary. Specifically, when $s=1$ and $a=1$,
(\ref{eq 1.1}) simplifies to the classical MEMS equation
\begin{equation}\label{eq 2024 1}
\left\{\arraycolsep=1pt
\begin{array}{lll}
-\Delta u = \frac{\lambda }{(1-u)^2}\quad & {\rm in}\quad\ \Omega,
\\[2mm]
 \phantom{\   }
 \quad\ u=0\quad &{\rm on}\quad\   \partial\Omega,
 \end{array}
\right.
\end{equation}
where parameter $\lambda>0$ characterizes the relative strength of the electrostatic and mechanical forces in the system,
equation (\ref{eq 2024 1}) models a simple electrostatic MEMS device consisting of a elastic membrane with boundary supported
at $0$ above a rigid ground plate located at $1$.
There exists a critical value $\lambda^*$ (pull-in voltage) such that if $\lambda \in (0,\lambda^*)$,
problem (\ref{eq 2024 1}) admits a minimal solution, while for $\lambda>\lambda^*$, no solution for  (\ref{eq 2024 1}) exists.
When $s=1$, the function  $a:\bar\Omega\to[0,1]$ is in the class of $C^\gamma(\Omega)\cap C(\bar\Omega)$ and satisfies
\begin{equation}\label{eq 2024 2}
a(x)\ge \kappa\rho(x)^\gamma,\quad \forall \ x\in\Omega
\end{equation}
with  $\kappa>0$, $\gamma\in(0,1)$, $\rho(x)={\rm dist}(x,\partial\Omega)$ for $x\in\Omega$, Chen-Wang-Zhou in \cite{CWZ1} solved
\begin{equation}\label{2024 3}
\left\{\arraycolsep=1pt
\begin{array}{lll}
 -\Delta    u = \frac{\lambda }{(a-u)^2}\quad  &{\rm in}\quad\ \Omega,
 \\[2mm]
 \phantom{- }
 0<u<a\quad &{\rm in}\quad\ \Omega,
 \\[2mm]
 \phantom{-\Delta   }
 u=0\quad &{\rm on}\quad   \partial \Omega,
\end{array}
\right.
\end{equation}
which  models a MEMS device that the static deformation  of the surface of membrane when it is applied voltage $\lambda$,
where $a$ is initially undeflected state  of the elastic membrane that contacts the ground plate on the boundary.
More results on MEMS see  \cite{CWZ2,GG, GHW,G, GPW, LBS, LY,P} and reference therein.

The primary objective of our paper is to examine the MEMS equation
(\ref{eq 1.1}) that incorporates the fractional Laplacian. In
pursuit of establishing the existence of minimal solutions to
(\ref{eq 1.1}), we proceed under the assumption that
\begin{enumerate}
\item[$(\cA_0)$]  For some $\gamma\in(0,s)$,
 \begin{equation}\label{1.1}
a(x)\ge \kappa\rho(x)^\gamma,\quad \forall \ x\in\Omega,
\end{equation}
where  $\kappa>0$ and
$$\rho(x)=\min\big\{\frac{1}{2}, {\rm dist}(x,\partial\Omega)\big\}. $$
\end{enumerate}

\smallskip

Now we show the existence result.

\begin{theorem}\label{teo 1}
Assume that  $a$   satisfies (\ref{1.1}) with $\gamma\in(0,\frac23s]$, then there exists a
finite pull-in voltage $\lambda^*$ such that
\begin{enumerate}
\item[$(i)$]   for $\lambda\in(0,\lambda^*)$, problem (\ref{eq 1.1}) admits a minimal solution $u_\lambda$ and
the mapping: $\lambda\mapsto u_\lambda$ is   strictly  increasing;

\item[$(ii)$]   for $\lambda>\lambda^*$, there is no solution for (\ref{eq 1.1});

\item[$(iii)$]    assume more that there exists $c \ge \kappa$ such that
\begin{equation}\label{a 1.0}
a(x)\le c\rho(x)^{\gamma}, \quad x\in\Omega,
\end{equation}
 then there exists  $\lambda_*\in(0, \lambda^*]$ such that for $\lambda\in(0,\lambda_*)$,    $u_\lambda\in H^s_0(\Omega)$ and
$$
{\rm for}\quad \gamma\not=\frac12s,\qquad \frac{\lambda}{c_0}\rho(x)^{\min\{s,2s-2\gamma\}}\le
u_\lambda(x)\le c_0 \lambda \rho(x)^{\min\{s,2s-2\gamma\}},\qquad \forall x\in
\Omega;
$$
$$
{\rm for}\quad\gamma=\frac12s,\qquad \frac{\lambda}{c_0}\rho(x)^s\ln \frac1{\rho(x)} \le
u_\lambda(x)\le c_0 \lambda \rho(x)^s\ln \frac1{\rho(x)},\qquad \forall x\in A_{\frac12},
$$
where $c_0\ge 1$ and $A_t=\{x\in\Omega:\ \rho(x)<t\}$ for $t>0$.

In particular, when $\Omega=B_1(0)$ and
$$a(x)=\kappa(1-|x|^2)^\gamma,\qquad \forall x\in B_1(0),$$
then the mappings: $\gamma\mapsto \lambda_*(\kappa,\gamma)$, $\gamma\mapsto\lambda^*(\kappa,\gamma)$ are decreasing.

\end{enumerate}
\end{theorem}

Theorem \ref{teo 1} affirms that the membrane of a MEMS device can
be tailored to match the surface of a unit semi-sphere. This implies
that
$$\Omega=B_1(0)\quad{\rm and}\quad a(x)=(1-|x|^2)^{\frac{s}{2}}.$$
Equivalently, in the scenario where $a(x)=\rho(x)^{\frac{s}{2}}$,
there exists a positive finite pull-in voltage $\lambda^*$. And from
Theorem \ref{teo 1}, we observe that the mapping $\lambda\mapsto
u_\lambda$ is increasing and uniformly bounded, let us denote
\begin{equation}\label{4.2}
 u_{\lambda^*}:=\lim_{\lambda\to\lambda^*} u_\lambda\quad{\rm in}\ \ \bar\Omega.
\end{equation}
We next show that $u_{\lambda^*}$ is a solution of (\ref{eq 1.1}) in the following weak sense.

\begin{definition}\label{def 1}
A function $u$ is a weak solution of (\ref{eq 1.1}) if $0\le u\le a$ and
$$
\int_\Omega u (-\Delta)^s \xi\, dx=\int_\Omega \frac{\lambda \xi}{(a-u)^2} dx,\qquad \forall \,\xi\in C_c^2(\Omega),
$$
where  $C_c^2(\Omega)$ is the space of all  $C^2$ functions with compact support in $\Omega$.

A solution (or weak solution) $u$ of (\ref{eq 1.1}) is stable (resp. semi-stable) if
$$
\int_\Omega |(-\Delta)^\frac s2 \varphi|^2dx=\int_\Omega  \varphi\cdot(-\Delta)^s \varphi dx
>\int_\Omega \frac{2\lambda \varphi^2}{(a-u)^3} dx,\quad ({\rm resp.}\ \ge)\quad \forall \, \varphi\in H^s_0(\Omega)\setminus\{0\}.
$$
\end{definition}

\begin{theorem}\label{teo 3}
Assume that $\lambda\in(0,\lambda^*)$, the function
$a\in C^\gamma(\Omega)\cap C(\bar\Omega)$  satisfies (\ref{1.1}) and (\ref{a 1.0}) with
 $c_0\ge \kappa>0$, $\gamma\in(0,\frac23s]$,
$u_\lambda$ is the minimal solution of (\ref{eq 1.1}) and $ u_{\lambda^*}$ is given by (\ref{4.2}).
Then

\begin{itemize}
  \item [$(i)$] $u_{\lambda^*}$ is a weak solution of (\ref{eq 1.1}) and
$u_{\lambda^*}\in W^{s,\frac{N}{N-\beta}}_0(\Omega)$ for any
$\beta\in(0,\min\{\gamma-s+1,s\})$;
  \item [$(ii)$] $u_\lambda$ is a stable solution of (\ref{eq 1.1}) with $\lambda\in(0,\lambda_*)$;
  \item [$(iii)$] if $\gamma=\frac23s$, we have that $\lambda^*=\lambda_*$,
  $ u_{\lambda^*}$ is a semi-stable weak  solution of (\ref{eq 1.1}).
\end{itemize}
Assume more that $1\le N\le \frac{14s}{3}$,  $\Omega=B_1(0)$ and
$a(x)=\kappa(1-|x|^2)^{\frac23s}$, then $ u_{\lambda^*}$ is a
classical solution of (\ref{eq 1.1}).
\end{theorem}

%Note that  the extremal solution $u_{\lambda^*}$ is a weak solution in special space. We show the stability of $u_\lambda$ %with $\lambda\in(0,\lambda_*)$
%by Hardy's inequality

%. Here we use the generalized Hardy's inequality  to arise the stability of $u_\lambda$ with $\lambda\in(0,\lambda_*)$,  for $1\le N\le \frac{14s}{3}$, %$\gamma=\frac23s$, the regularity of $u_{\lambda^*}$ could be improved into classical sense.

For the case of $\frac23s<\gamma<s$, the nonexistence result is stated as follows.

\begin{theorem}\label{teo 2}
Assume that $a\in  C(\bar\Omega)$ satisfies $0<a(x)\leq c \, \rho(x)^\gamma$ with $\gamma\in(\frac23s,s)$ and $c>0$.
Then problem (\ref{eq 1.1}) admits no nonnegative solution for any $\lambda>0$.

\end{theorem}

\begin{remark}
\begin{itemize}
  \item [$(i)$] Our model is derived from the enclosed MEMS model \cite{CWZ1}.
The incorporation of the fractional Laplacian, a key component in
our model, is notable for its applicability in modeling fractional
quantum mechanics. This is particularly evident in the examination
of particles maneuvering through stochastic fields, which are often
represented by L\'evy processes \cite{L,L1}. Furthermore, the
fractional Laplacian is construed as a Pseudo-Relativistic operator
\cite{A,JKS} and see \cite{DT1,TV1} for the square root of the Laplacian.

  \item [$(ii)$] The most difficulty in the fractional case is the estimates
of $(-\Delta)^s \rho^\tau(x)$ for $\tau\in(0,2s)$, where $\rho$ is
distance to the boundary. Different from the boundary blowing up
case $\tau\in(-1,0)$ in \cite{CFQ}, they have different monotonicity
near the boundary. See the proof of Proposition \ref{pr 2.1.0}
below.

\item [$(iii)$] For  $\gamma\in(\frac23s,s)$, the nonexistence in Theorem
\ref{teo 2} shows that the pull-in voltage $\lambda^*=0$, so there
is a jump for $\lambda^*$ at $\gamma=\frac23s$, since $\lambda^*>0$
for $\gamma=\frac23s$.

\item [$(iv)$] Our next program is to develop our observations to study  the related parabolic equation
\begin{equation}\label{eq 1.1-p}
\left\{\arraycolsep=1pt
\begin{array}{lll}
\partial_t v+ (-\Delta)^s   v = \frac{\lambda }{(a-v)^2}\quad  &{\rm in}\quad\ (0,T)\times \Omega,
 \\[2mm]
 \phantom{\partial_t v+-\ }
 0<v< a\quad &{\rm in}\quad\ (0,T)\times\Omega,
 \\[2mm]
 \phantom{\partial_t v+--  }
 \quad\ v=0\quad &{\rm in}\quad\   (0,T)\times\big(\mathbb{R}^N\backslash \Omega\big), \\[2mm]
  \phantom{\partial_t v+-  }
  v(0,\cdot)=v_0\quad &{\rm in}\quad\     \Omega,
\end{array}
\right.
\end{equation}
where $T\in(0,+\infty]$ and $v_0$ is an initial data.
\end{itemize}
\end{remark}

The structure of this article is as follows: In Section 2, we
provide estimates crucial for determining the pull-in voltage.
Section 3 is dedicated to establishing the existence of a pull-in
voltage denoted as $\lambda^*$, such that problem (\ref{eq 1.1})
possesses a minimal solution for $\lambda\in(0,\lambda^)$. We also
analyze the boundary decay of this minimal solution and present the
proof of Theorem \ref{teo 2}. Section 4 focuses on estimating
$\lambda_*$ and $\lambda^*$ when $\Omega=B_1(0)$ and
$a(x)=\kappa(1-|x|)^\gamma$, along with providing the proof of
Theorem \ref{teo 1}. Finally, in Section 5, we investigate the
regularity and stability properties of $u_{\lambda^*}$ and offer the
proof of Theorem \ref{teo 3}.

\setcounter{equation}{0}
\section{Preliminary}

This section is devoted to the estimates of boundary behavior and
the fractional Green operator
$\mathbb{G}_{s,\Omega}[\rho^{\tau-2s}]$.
 To this end, we give some notations as follows.
Let $\delta_0\in(0,\frac12)$ be  such that the distance function
$\rho(\cdot)$ is of class $C^2$ in $A_{\delta_0}:=\{x\in \Omega,\rho(x)<\delta_0\}$. For $\tau\in(0,2s)$, we
define
\begin{equation}\label{3.3.1}
V_\tau(x)=\left\{ \arraycolsep=1pt
\begin{array}{lll}
 h_\tau(x),\ \ \ \ &
x\in \Omega\setminus A_{\delta_0},\\[2mm]
 \rho(x)^{\tau},\ & x\in A_{\delta_0},\\[2mm]
0,\ &x\in\Omega^c
\end{array}
\right.
\end{equation}
and
\begin{equation}\label{3.3.1a}
W_s(x)=\left\{ \arraycolsep=1pt
\begin{array}{lll}
 h_s(x),\ \ \ \ &
x\in \Omega\setminus A_{\delta_0},\\[2mm]
 \rho(x)^{s} \ln \frac1{\rho(x)},\ \ \, & x\in A_{\delta_0},\\[2mm]
0,\ &x\in\Omega^c,
\end{array}
\right.
\end{equation}
 where  the function $h_\tau$  is  positive  such that
$V_\tau$ is $C^2$ in $\Omega$.  We have the following estimates.
\begin{proposition}\label{pr 2.1.0}
Let $\Omega$ be a connected bounded open subset of $\R^N$ with $C^2$ boundary,  $s\in(0,1)$
the functions $V_\tau$ and $W_s$ be given by (\ref{3.3.1}) and (\ref{3.3.1a}), respectively. Then
 there exist  $\delta_1\in (0,\delta_0)$ and  $c_1>1$ depends on $\tau$ such that
\begin{itemize}
  \item [$(i)$] if $\tau\in(0,s)$, then
$$
\frac1{c_1} \rho(x)^{\tau-2s }\leq (-\Delta)^s V_\tau(x)\leq
{c_1}\rho(x)^{\tau-2s}, \ \ \ \forall \, x\in A_{\delta_1}.$$

  \item [$(ii)$] if
 $\tau\in(s,2s)$, then
$$
\frac1{c_1}\rho(x)^{\tau-2s }\leq-(-\Delta)^s V_\tau(x)\leq
{c_1}\rho(x)^{\tau-2s}, \ \ \ \forall \, x\in A_{\delta_1}.$$
  \item [$(iii)$] if
 $\tau=s$, then
$$
|(-\Delta)^s V_\tau (x)|\leq {c_1}(\tau), \ \ \ \forall \, x\in
A_{\delta_1}.$$

\item [$(iv)$]
$$
\frac1{c_1}\rho(x)^{-s}\leq-(-\Delta)^s W_s(x)\leq
{c_1}\rho(x)^{-s}, \ \ \ \forall \, x\in A_{\delta_1}.$$
\end{itemize}
\end{proposition}

To prove Proposition \ref{pr 2.1.0}, we need the following auxiliary lemma.

\begin{lemma}\label{lm 2.1}
For $\tau\in (0,2s)$, let
\begin{equation}\label{3.1.4}
\psi_s(\tau)=\int^{+\infty}_{0}\frac{2-(1+t)^{\tau} -|1-t|^{\tau}\chi_{(0,1)}(t)}{t^{1+2s}}dt,
\end{equation}
 where $\chi_{A}$ is the characteristic function of  the set $A$.

Then the function $\psi_s$ is strictly concave in $(0,2s)$ and
$$\psi_s(s)=0,\qquad  \psi_s(\tau)>0\ {\rm for}\ \tau\in (0,s), \qquad \psi_s(\tau)<0\ {\rm for}\ \tau\in (s,2s).$$
Moreover,
$$\psi_s'(s)=-\int^{+\infty}_{0}\frac{ (1+t)^s\log (1+t)+|1-t|^s \chi_{(0,1)}(t) \log |1-t| }{t^{1+2s}}dt<0. $$
 \end{lemma}
\noindent{\bf Proof.} Denote
$$g_\tau(t)=\left\{\arraycolsep=1pt
\begin{array}{lll}
t^\tau,  \quad & t>0,
 \\[2mm]
 \phantom{ }
 0,   \quad  &  t\leq 0,
\end{array}
\right.$$
by direct computation,  for $l>0$, we have that
 \begin{align*}
(-\Delta)^s_{\R}\, g_\tau(l)&=\frac{c_{1,s}}{2}  \int_{\R}\frac{2l^{\tau}-g_\tau(l+t)-g_\tau(l-t)}{t^{1+2s}}dt
\\&=c_{1,s}  \int_{0}^\infty\frac{2l^{\tau}- (l+t)^\tau - |l-t|^\tau\chi_{(0,l)}(t) }{t^{1+2s}}dt
\\&= c_{1,s}l^{\tau-2s}   \int^{+\infty}_{0}\frac{2-(1+t)^{\tau}-|1-t|^{\tau} \chi_{(0,1)}(t)}{t^{1+2s}}dt= c_{1,s}l^{\tau-2s} \psi_s(\tau).
\end{align*}
  It is shown in \cite[Theorem 5.1]{CLO} that  $g_s$ is a solution of $(-\Delta)^{s}_{\R}\, g_s=0$ in $(0,+\infty)$, which is equivalent that $\psi_s(s)=0$.
  Moreover, it holds
$$\psi_s(0)=\int^{+\infty}_{1}\frac{1}{t^{1+2s}}dt=\frac1{2s}>0,$$
$$\psi_s'(\tau)=-\int_0^{+\infty}\frac{(1+t)^{\tau}\log(1+t)+ |1-t|^{\tau}\chi_{(0,1)}(t)\log|1-t|}{t^{1+2s}}dt$$
and
$$\psi_s''(\tau)=-\int_0^{+\infty}\frac{(1+t)^{\tau}[\log(1+t)]^2 + |1-t|^{\tau}[\chi_{(0,1)}(t)\log|1-t|]^2}{t^{1+2s}}dt<0,$$
which implies that $\psi_s$ is strictly concave.
Combining $\psi_s(s)=0$ with $\psi_s(0)>0$,
we observe that $s$ is the unique zero point of $\psi_s$ in $(0,2s)$,  $\psi_s>0$ in $(0,s)$, $\psi_s<0$ in $(s,2s)$ and $\psi_s'(s)<0$.
   \hfill$\Box$

\bigskip

Now we are ready to give the proof of Proposition \ref{pr 2.1.0}.

\smallskip

\noindent{\bf Proof of Proposition \ref{pr 2.1.0}.} For $\tau\in(0,2s)$,  we adopt the  arguments in \cite[Proposition 3.1]{CFQ},
which dealt with the case $\tau\in(-1,0)$.  However, the big difference is the monotonicity for $\tau>0$,
so we provide the details of the proof here and also show the estimate for critical case $\tau=s$.

 Since the boundary of $\Omega $ is compact, we only need to show the corresponding inequality holds
in a neighborhood of any point $\bar x\in\partial\Omega$, without loss of generality, we may assume that $\bar x=0$. Given $0<\eta\le \delta$, let
$$
Q_\eta=\{z=(z_1,z')\in\R\times\R^{N-1}:\  |z_1|<\eta,|z'|<\eta\}$$
and  $Q_\eta^+=\{z\in Q_\eta:\  z_1>0\}$.
Let $\varphi:\R^{N-1}\to\R$  be a $C^2$
function such that $(z_1, z')\in \Omega\cap Q_\eta$ if and only if $z_1\in (\varphi(z'),\eta)$,
 moreover,   $(\varphi(z'),z')\in\partial\Omega$ for all $|z'|<\eta$. We further assume that  $(-1,0,\cdot\cdot\cdot,0)$ is the outer normal
vector of $\Omega$ at $\bar x$.

In the proof of our inequalities, we let $x=(x_1,0)$, with $x_1\in(0,\eta/4)$, be a generic point in $A_{\eta/4}$.
Then $|x-\bar x|=\rho(x)=x_1$. Denote $\delta(V_\tau,x,y)=2V_\tau(x)-V_\tau(x+y)-V_\tau(x-y)$,
we have that
\begin{eqnarray*}
\frac2{c_{N,s}}  (-\Delta)^s V_\tau(x)=  \int_{Q_{
\eta}}\frac{\delta(V_\tau,x,y)}{|y|^{N+2s}}dy+  \int_{\R^N\setminus
Q_{ \eta}}\frac{\delta(V_\tau,x,y)}{|y|^{N+2s}}dy
\end{eqnarray*}
and
\begin{equation*}
\Big|\int_{\R^N\setminus
Q_{ \eta}}\frac{\delta(V_\tau,x,y)}{|y|^{N+2s}}dy\Big|
\le c_2,
\end{equation*}
where the constant $c_2$ is independent of $x$. Now we do the estimate for
\begin{equation}\label{cotadelta1}
\cE(x_1):=-\int_{Q_{
\eta}}\frac{\delta(V_\tau,x,y)}{|y|^{N+2s}}dy.
\end{equation}
We divide the proof into two steps.

\smallskip

{\bf Step 1:  Lower bounds for $\cE(x_1)$.}
It follows by \cite[Lemma 3.1]{CFQ}   that
there exist $\eta\in(0,\delta_0]$ and $c_3>0$ such that
$$
\rho(z)\ge (z_1-\varphi(z'))(1-c_3|z'|^2),\quad  \ \forall \, z=(z_1,z')\in  Q_\delta\cap\Omega
$$
and then
\begin{equation}\label{1.3.5}
\rho(z)^\tau\ge |z_1-\varphi(z')|^\tau (1-c_3|z'|^2)^{\tau}.
\end{equation}

%When $0<\eta\le \delta_0/2$, for
For $y\in Q_{\eta}$, we note that
 $x\pm y\in Q_\delta $ and then $x\pm y\in Q_\delta\cap\Omega$ if and only if
$\varphi(\pm y')<x_1\pm y_1<\delta$ and $|y'|<\delta$.
Combining with (\ref{1.3.5}), there exist $0<C_0<C_1$ such that
\begin{align*}
V_\tau(x+ y)
 \geq \rho(x+ y)^\tau &\ge\big(x_1+ y_1-\varphi(y')\big)^\tau (1-c_3|y'|^2)^{\tau}
\\& \geq  \big(x_1+ y_1-\varphi(y')\big)^\tau (1-c_4\tau |y'|^2), \qquad  x+y\in Q_{\delta}\cap\Omega
\end{align*}
and
$$
  V_\tau(x- y)=\rho(x- y)^\tau\ge\big(x_1- y_1-\varphi(y')\big)^\tau (1-c_4\tau |y'|^2), \qquad x-y\in Q_{\delta}\cap\Omega.
$$

For $y\in Q_{\eta}$, if $x\pm y\in Q_\delta\cap\Omega^c$, by definition of $V_\tau$, it holds that
$V_\tau(x\pm y)=0.$

To be convenient for the analyze, let us denote
\begin{equation*}
I_+=(\varphi( y')-x_1,\eta-x_1), \  \qquad I_-=(x_1-\eta,x_1-\varphi(-y'))
\end{equation*}
and the functions
\begin{align*}
& {\bf I}(y) = \chi_{_{I_+}}(y_1)|x_1+y_1-\varphi(y')|^{\tau}+\chi_{_{I_-}}(y_1)|x_1-y_1-\varphi(-y')|^{\tau} -2x_1^\tau,
\\[1mm]
%\tilde I(y)&=&\chi_{(x_1-\eta,x_1)}(y_1)|x_1-y_1-\varphi(-y')|^{\tau}+\chi_{(-x_1,\eta-x_1)}(y_1)|x_1+y_1-\varphi(y')|^
%{\tau}-2x_1^\tau
%\label{I2}\\
% \mbox{and}\quad  & &\nonumber \\
& {\bf J}(y_1) = \chi_{_{(x_1-\eta,x_1)}}(y_1)|x_1-y_1|^{\tau} +\chi_{_{(-x_1,\eta-x_1)}}(y_1)|x_1+y_1|^{\tau}-2x_1^\tau, \\[1mm]
& {\bf I}_1(y) = \big( \chi_{_{I_+}}(y_1) -\chi_{_{(-x_1,\eta-x_1)}}(y_1)\big) |x_1+y_1|^\tau,
\\[1mm]
& {\bf I}_{-1}(y) = \big( \chi_{_{I_-}}(y_1) -\chi_{_{(x_1-\eta,x_1)}}(y_1)\big) |x_1-y_1|^\tau,
\\[1mm]
&  {\bf I}_2(y) =  \chi_{_{I_+}}(y_1) \big(|x_1+y_1-\varphi(y')|^\tau - |x_1+y_1|^\tau\big),
\\[1mm]
& {\bf I}_{-2}(y) =   \chi_{_{I_-}}(y_1) \big(|x_1-y_1-\varphi(-y')|^\tau - |x_1-y_1|^\tau\big),
 \end{align*}
where $\chi_{_A}$ denotes the characteristic function of the set $A$. Observe that
\begin{eqnarray}\label{Ex1}
   \cE(x_1)\ge \int_{Q_{\eta}} \frac{{\bf I}(y)}{|y|^{N+2s}}dy-E_0(x_1)
   =\int_{Q_{\eta}} \frac{{\bf J}(y_1)}{|y|^{N+2s}}dy-E_0(x_1)+E_1(x_1)+E_2(x_1),
\end{eqnarray}
with
\begin{equation*}\label{E0}
E_0(x_1)=\tau c_4 \int_{Q_\eta} \frac{\big( |x_1+y_1-\varphi(y')|^{\tau}+ |x_1-y_1-\varphi(y')|^{\tau}\big)|y'|^2 }{|y|^{N+2s}}dy
\end{equation*}
  and
\begin{equation}\label{Ei}
E_i(x_1)=\int_{Q_{\eta}} \frac{{\bf I}_i(y)+{\bf I}_{-i}(y)}{|y|^{N+2s}}dy,\quad i=1,2.
\end{equation}
Then
 \begin{eqnarray*}\label{A10}
 0\leq  E_0(x_1)  \leq 2c_5\int_{Q_{\eta}}\frac{|y'|^2\eta^\tau}{|y|^{N+2s}}dy  \leq   c_6\eta^\tau \int_{Q_{\eta}} \frac{1 }{|y|^{N+2s-2}}dy
 \leq   c_7\eta^{\tau+2-2s}.
\end{eqnarray*}

Note that
$$
\int_{Q_{\eta}}\frac{{\bf J}(y_1)}{|y|^{N+2s}}dy
=x_1^{\tau-2s}\int_{Q_{\frac{\eta}{x_1}}}    \frac{{\bf J}(x_1z_1)x_1^{-\tau}}{|z|^{N+2s}}dz= 2x_1^{\tau-2s}({\bf R}_1-{\bf R}_2),
$$
where
$$
{\bf R}_1=\int_{Q_{\frac{\eta}{x_1}}^+}\frac{\chi_{_{(0,1)}}(z_1)|1-z_1|^{\tau}
+(1+z_1)^{\tau}-2}{|z|^{N+2s}}dz
$$
and
$$
{\bf R}_2=\int_{Q_{\frac{\eta}{x_1}}^+}
\frac{\chi_{_{(\frac{\eta}{x_1}-1,\frac{\eta}{x_1})}}(z_1)(1+z_1)^{\tau}}{|z|^{N+2s}}dz.
$$

Since
 \begin{align*}
&\quad\  \int_{\R^N_+}\frac{\chi_{_{(0,1)}}(z_1)|1-z_1|^{\tau}+(1+z_1)^{\tau}-2}{|z|^{N+2s}}dz\\
&= \int_0^{+\infty}\frac{\chi_{_{(0,1)}}(z_1)|1-z_1|^{\tau}+(1+z_1)^{\tau}-2}{z_1^{1+2s}}dz_1
\int_{\R^{N-1}}\frac{1}{(|z'|^2+1)^{\frac{N+2s}{2}}}dz'
\\
&=- c_8\,\psi_s(\tau)
\end{align*}
and
 \begin{align*}
 \int_{(Q_{\frac{\eta}{x_1}}^+)^c}\frac{\chi_{_{(0,1)}}(z_1)|1-z_1|^{\tau}
+(1+z_1)^{\tau}-2}{|z|^{N+2s}}dz =c_9
{x_1^{2s-\tau}}(1+ o(1)),
\end{align*}
then
 \begin{eqnarray}\label{R1}
 {\bf R}_1=-c_8 \psi_s(\tau)-
{c_{9}x_1^{2s-\tau}- o(x_1^{2s-\tau})}.
\end{eqnarray}
On the other hand, by direct computation, it yields that
\begin{eqnarray} \label{R2}
\quad {\bf R}_2
&=&\int_{\frac{\eta}{x_1}-1}^{\frac{\eta}{x_1}}\frac{(1+z_1)^{\tau}}{z_1^{1+2s}}
\int_{\tilde{B}_{\frac{\eta}{x_1}}}\frac1{(1+|z'|^2)^{\frac{N+2s}2}}dz'dz_1\le
%\\&\le&C\int_{\frac{\eta}{x_1}-1}^{\frac{\eta}{x_1}}z_1^{\tau-2s-1}dz_1
%\le
c_{10}x_1^{2s-\tau+1},
\end{eqnarray}
here and in what follows we denote by $B_\sigma$ the ball of radius $\sigma$ in $\R^{N-1}$.
From \equ{R1} and \equ{R2}, we obtain that
\begin{eqnarray*}\label{EQ}
\int_{Q_{\eta}}\frac{{\bf J}(y_1)}{|y|^{N+2s}}dy=c_{11}x_1^{\tau-2s}
{\big(-\psi_s(\tau)+c_{12}x_1^{2s-\tau}+ o(x_1^{2s-\tau})\big)}.
\end{eqnarray*}

 Now we   consider the  term ${\bf I}_1(y)$, the estimate for ${\bf I}_{-1}(y)$ is similar.
Observe that
\begin{align*}
 \int_{Q_{\eta}} \frac{{\bf I}_1(y)}{|y|^{N+2s}}dy
=
-\int_{\tilde{B}_\eta}\int^{\varphi(y')-x_1}_{-x_1}\frac{|x_1+y_1|^\tau}{|y|^{N+2s}}dy_1dy'
 = -x_1^{\tau-2s} F_1(x_1),\label{F1}
%\\
%\nonumber&\ge& -\int_{B_\eta}\int^{\varphi_+(y')-x_1}_{-x_1}\frac{|x_1+y_1|^\tau}{|y|^{N+2s}}dy_1dy'\\
%&=& -x_1^{\tau-2s}F_1(x_1),\label{F1}
\end{align*}
where
\begin{eqnarray}\label{intF1}
 F_1(x_1)= \int_{\tilde{B}_\frac{\eta}{x_1}}\int^{\frac{\varphi(x_1z')}{x_1}}_{0}
\frac{|z_1|^\tau}{((z_1-1)^2+|z'|^2)^{(N+2s)/2}}dz_1dz'.
\end{eqnarray}
Let $\varphi_-(y')=\min\{\varphi(y'),0\}$ and $\varphi_+(y')=\varphi(y')-\varphi_-(y')$.
Note that $0\le \varphi_+(y')\le c_{13}|y'|^2$ for $|y'|\le \eta$, for some $(z_1,z')$ satisfying
 $0\le z_1\le \frac{\varphi_+(x_1 z')}{x_1}$ and $|z'|\le\frac{\eta}{x_1}$, then
\begin{eqnarray*}
\label{cota}(1-z_1)^2+|z'|^2\ge \frac{1}{4}(1+|z'|^2).
\end{eqnarray*}
Thus,
\begin{align*}
 F_1(x_1)&\le c_{14}\int_{\tilde{B}_\frac{\eta}{x_1}}\int^{\frac{\varphi_+(x_1z')}{x_1}}_{0}
\frac{|z_1|^\tau}{(1+|z'|^2)^{(N+2s)/2}}dz_1dz'\\
&\le   c_{15}x_1^{\tau+1}   \int_{\tilde{B}_\frac{\eta}{x_1}} \frac{|z'|^{2(\tau+1)}}{(1+|z'|^2)^{(N+2s)/2}}dz' \\
&\le  c_{16}x_1^{\tau+1}(x_1^{-2\tau+2s-1}+1)\le c_{17}x_1^{\min\{\tau+1,2s-\tau\}},
\end{align*}
so
\begin{equation}\label{E1(x1)}
E_1(x_1)\ge -c_{18}x_1^{\tau-2s}x_1^{\min\{\tau+1,2s-\tau\}}.
\end{equation}

Now we do the estimate  of $E_2(x_1)$.
Since
\begin{align*}
\int_{Q_{\eta}} \frac{{\bf I_2}(y)}{|y|^{N+2s}}dy&= \int_{\tilde{B}_\eta}\int_{\varphi(y')-x_1}^{\eta-x_1}\frac{|x_1+y_1-\varphi(y')|^{\tau}-|x_1+y_1|^{\tau}}
{(y_1^2+|y'|^2)^{\frac{N+2s}2}}dy_1dy' \nonumber \\
&\ge \int_{\tilde{B}_\eta}\int_{\varphi_+(y')-x_1}^{\eta-x_1}\frac{|x_1+y_1-\varphi_+(y')|^{\tau}-|x_1+y_1|^{\tau}}
{(y_1^2+|y'|^2)^{\frac{N+2s}2}}dy_1dy'\nonumber \\
&= \int_{\tilde{B}_\eta}\int_{\varphi_+(y')}^{\eta}\frac{|z_1-\varphi_+(y')|^{\tau}-|z_1|^{\tau}}
{((z_1-x_1)^2+|y'|^2)^{\frac{N+2s}2}}dz_1dy'\nonumber \\
&\ge \int_{\tilde{B}_\eta}\int_{0}^{\eta}\frac{|z_1-\varphi_+(y')|^{\tau}-|z_1|^{\tau}}
{((z_1-x_1)^2+|y'|^2)^{\frac{N+2s}2}}dz_1dy' \nonumber \\
&  + \int_{\tilde{B}_\eta}\int_{\varphi_+(y')}^{0}\frac{-|z_1|^{\tau}}
{((z_1-x_1)^2+|y'|^2)^{\frac{N+2s}2}}dz_1dy'\nonumber
\\[1mm] &=
E_{21}(x_1)+E_{22}(x_1).
\end{align*}
Note that $E_{22}(x_1)$ is similar to $ F_1(x_1)$. For the estimate of $E_{21}(x_1)$, we use integration by parts to obtain
\begin{align*}
E_{21}(x_1)&=
\frac1{\tau+1}\int_{\tilde{B}_\eta} \left\{\frac{(\eta-\varphi_+(y'))^{\tau+1}-\eta^{\tau+1}}
{((\eta-x_1)^2+|y'|^2)^{\frac{N+2s}2}}    -\frac{(-\varphi_+(y'))^{\tau +1}}{(x_1^2+|y'|^2)^{\frac{N+2s}{2}}} \right\}  dy'\\
&\quad +
\frac{N+2s}{\tau+1}\int_{\tilde{B}_\eta}\int_{0}^{\eta}\frac{(z_1-\varphi_+(y'))^{\tau+1}-(z_1)^{\tau+1} }
{((z_1-x_1)^2+|y'|^2)^{\frac{N+2s}2+1}}(z_1-x_1)dz_1dy'\\
&=  A_1+A_2.
\end{align*}
For the first integral, we have that
\begin{align*}
A_1&\ge
\frac1{\tau+1}\int_{\tilde{B}_\eta} \left\{\frac{-\eta^{\tau+1}}
{((\eta-x_1)^2+|y'|^2)^{\frac{N+2s}2}}    -\frac{(-\varphi_+(y'))^{\tau +1}}{(x_1^2+|y'|^2)^{\frac{N+2s}{2}}} \right\}  dy'
\\
&\ge
-c_{19}-c_{20}\int_{\tilde{B}_\eta}\frac{|y'|^{2\tau+2}}{(x_1^{2}+|y'|^2)^{\frac{N+2s}2}}dy'
\\
& \ge -c_{21}x_1^{\tau-2s+\tau+1} -\bar c_{19}.
\end{align*}
For the second integral, since $\tau\in(0,2s)$ and
$(z_1-\varphi_+(y'))^{\tau+1} -|z_1|^{\tau+1}<0$,  we have that
\begin{align*}%\label{q 3.24}
A_2&\ge \frac{N+2s}{\tau+1}\int_{\tilde{B}_\eta}\int^{\eta}_{x_1}\frac{(z_1-\varphi_+(y'))^{\tau+1}
-|z_1|^{\tau+1}}{((z_1-x_1)^2+|y'|^2)^{\frac{N+2s}2+1}}(z_1-x_1)dz_1dy'
\\[4mm]&\ge (N+2s)\int_{\tilde{B}_\eta}\int^{\eta}_{x_1}\frac{-\varphi_+(y')z_1^\tau}{((z_1-x_1)^2+|y'|^2)^{\frac{N+2s}2+1}}(z_1-x_1)dz_1dy'
\\[4mm]&\ge
{-c_{22}x_1^{\tau-2s+1}}\int_{\tilde{B}_{\eta/x_1}}\int^{\frac{\eta}{x_1}}_1\frac{|z'|^2z_1^\tau}{((z_1-1)^2+|z'|^2)^{\frac{N+2s}2+1}}(z_1-1)dz_1dz'
\\[4mm]&\ge  {-c_{23}x_1^{\tau-2s+1}x_1^{2s-\tau}}=-c_{23} x_1,
\end{align*}
where $c_{22},c_{23}>0$ independent of $x_1$ and the second inequality holds by the fact that
$a^{\tau+1}-b^{\tau+1}\geq(\tau+1)(a-b)b^\tau$ for
$a>0,b\geq0$.
Then
\begin{equation*}
E_2(x_1)\ge -c_{24}x_1^{\tau-2s}x_1^{\min\{\tau+1,2s-\tau\}}.
\end{equation*}

As a consequence, we obtain that
\begin{equation*}
\cE(x_1)\ge c_{25}x_1^{\tau-2s} \Big[ -\psi_s(\tau)-c_{26}x_1^{2s-\tau} -c_{27}x_1^{\min\{\tau+1,2s-\tau\}} )\big)\Big] .
\end{equation*}

{\bf Step 2:  Upper bounds for $\cE(x_1)$.}
 Note that
$$
\rho(z)\le  z_1-\varphi(z'),  \qquad\ \forall \, z= (z_1,z')\in Q_\eta\cap \Omega
$$
and for $\tau\in(0,2s)$
\begin{equation*}
\rho^\tau(z)\le (z_1-\varphi(z'))^\tau, \qquad\ \forall \, z= (z_1,z')\in Q_\eta\cap \Omega.
\end{equation*}
Then, for $x\pm y\in Q_\eta\cap \Omega$, we have that
\begin{eqnarray*}
V_\tau(x\pm y)=\rho(x\pm y)^\tau\le (x_1\pm
y_1-\varphi(\pm y'))^\tau .
\end{eqnarray*}
Therefore,
\begin{eqnarray*}
 \cE(x_1)
 \le \int_{Q_\eta}\frac{{\bf I}(y)}{|y|^{N+2s}}dy =\int_{Q_\eta}\frac{{\bf J}(y)}{|y|^{N+2s}}dy+E_1(x_1)+E_2(x_1),
\end{eqnarray*}
where ${\bf I}$,  ${\bf J}$,  $E_1$ and $E_2$ are given in Step 1.

 Note that $0\le \varphi_+(y')\le c_{28}|y'|^2$ for $|y'|\le \eta$. For $z=(z_1,z')$ satisfying that
 $0\le z_1\le \frac{\varphi_+(x_1 z')}{x_1}$ and $|z'|\le\frac{\eta}{x_1}$, we have that
$(1-z_1)^2+|z'|^2\leq c_{29}(1+|z'|^2)$, also denote $F_1$ as in \equ{intF1},  we have that
$$
F_1(x_1)\ge c_{30} \int_{\tilde{B}_\frac{\eta}{x_1}}\int^{\frac{\varphi_-(x_1z')}{x_1}}_{0}
\frac{|z_1|^\tau}{(1+|z'|^2)^{(N+2s)/2}}dz_1dz'
$$
 and then
\begin{equation*}
E_1(x_1)\le c_{31}x_1^{\tau-2s}x_1^{\min\{\tau+1, {2s-\tau}\}}.
\end{equation*}
We next do the estimate for $E_2(x_1)$,  we
first consider the term ${\bf I}_2(y)$:
\begin{align*}
\int_{Q_{\eta}} \frac{{\bf I}_2(y)}{|y|^{N+2s}}dy&
\le \int_{\tilde{B}_\eta}\int_{\varphi_-(y')}^{\eta}\frac{|z_1-\varphi_-(y')|^{\tau}-|z_1|^{\tau}}
{((z_1-x_1)^2+|y'|^2)^{\frac{N+2s}2}}dz_1dy' =\tilde E_{21}(x_1).
\end{align*}
Using integration by parts, we have that
\begin{align*}
  \tilde E_{21}(x_1)& =
   \frac1{\tau+1}\int_{\tilde{B}_\eta} \left\{\frac{|\eta-\varphi_-(y')|^{\tau+1}-\eta^{\tau+1}}
{((\eta-x_1)^2+|y'|^2)^{\frac{N+2s}2}}    -\frac{|\varphi_-(y')|^{\tau +1}}{((\varphi_-(y')-x_1)^2+|y'|^2)^{\frac{N+2s}{2}}} \right\}  dy'\\
&  \quad +
\frac{N+2s}{\tau+1}\int_{\tilde{B}_\eta}\int_{\varphi_-(y')}^{\eta}\frac{(z_1-\varphi_-(y'))^{\tau+1}-z_1^{\tau+1} }
{((z_1-x_1)^2+|y'|^2)^{\frac{N+2s}2+1}}(z_1-x_1)dz_1dy'\\
&  \le  {\frac1{\tau+1}\int_{\tilde{B}_\eta} \frac{(\eta-\varphi_-(y'))^{\tau+1}-\eta^{\tau+1}}
{((\eta-x_1)^2+|y'|^2)^{\frac{N+2s}2}}   dy'}\\
&\quad +\frac{N+2s}{\tau+1}\int_{\tilde{B}_\eta}\int_{x_1}^{\eta}\frac{ (z_1-\varphi_-(y'))^{\tau+1}-z_1^{\tau+1}   }
{((z_1-x_1)^2+|y'|^2)^{\frac{N+2s}2+1}}(z_1-x_1)dz_1dy'
\end{align*}
and then
\begin{equation*}
 {E_2(x_1)\le c_{32}x_1^{\tau-2s}x_1^{\min\{\tau+1,2s-\tau\}}.}
\end{equation*}

As a consequence,
\begin{equation*}
\cE(x_1)\le c_{33}x_1^{\tau-2s} \Big[ -\psi_s(\tau)+c_{34}x_1^{2s-\tau} +c_{35}x_1^{\min\{\tau+1,2s-\tau\}} \Big]
\end{equation*}
where $c_{33},c_{34},c_{35}>0$ independent of $x_1$.

\medskip

When $\tau \in (s,2s)$, by the fact that  $ \psi_s(\tau)<0$, combining Step 1 with Step2,  it holds that
$$
\frac{-\psi_s(\tau)}{2}  x_1^{\tau-2s}\leq  -(-\Delta)^s V_\tau(x)\leq   -2\psi_s(\tau)x_1^{\tau-2s}.
$$

When $\tau \in (0,s)$, we know that $ \psi_s(\tau)>0$, combining Step 1 with Step2,  it is true that
$$
   \frac{ \psi_s(\tau)}{2} x_1^{\tau-2s}\leq   (-\Delta)^s V_\tau(x)\leq 2\psi_s(\tau) x_1^{\tau-2s}.
$$

\medskip

When $\tau=s$,  since  $ \psi_s(\tau)=0$, we have that
\begin{eqnarray*}
 |(-\Delta)^s V_\tau (x)|\leq  c_{36}(\tau).
\end{eqnarray*}

Finally, we show the  bounds for $(-\Delta)^s W_s$. By the definition of the fractional Laplacian, we have that
\begin{align*}
\frac2{c_{N,s}}  (-\Delta)^s W_s(x)=  \int_{Q_{
\eta}}\frac{\delta(W_s,x,y)}{|y|^{N+2s}}dy+  \int_{\R^N\setminus
Q_{ \eta}}\frac{\delta(W_s,x,y)}{|y|^{N+2s}}dy
\end{align*}
and
\begin{align*}
\Big|\int_{\R^N\setminus
Q_{ \eta}}\frac{\delta(W_s,x,y)}{|y|^{N+2s}}dy\Big|\le c_{37}.
\end{align*}
Denote
$$\cF(x_1)=-\int_{Q_{
\eta}}\frac{\delta(W_s,x,y)}{|y|^{N+2s}}dy. $$

When $x+y\in Q_{\delta}\cap\Omega$, we have that
\begin{align*}
W_s(x+ y)
 \geq \rho(x+ y)^s\ln\frac1{\rho(x+y)} &\ge\big(x_1+ y_1-\varphi(y')\big)^s\ln\frac1{x_1+ y_1-\varphi(y')}  (1-c_{38}|y'|^2)^s
\\& \geq  \big(x_1+ y_1-\varphi(y')\big)^s\ln\frac1{x_1+ y_1-\varphi(y')} (1-c_{39}  |y'|^2s)
\end{align*}
and
$$
W_s(x- y) \ge\big(x_1- y_1-\varphi(y')\big)^s \ln\frac1{x_1- y_1-\varphi(y')} (1-c_{39}  |y'|^2s) , \quad x-y\in Q_{\delta}\cap\Omega.
$$

  For $y\in Q_{\eta}$, if $x\pm y\in Q_\delta\cap\Omega^c$, by the definition of $W_s$, it holds that
$W_s(x\pm y)=0.$

Let
\begin{equation*}
I_+=(\varphi( y')-x_1,\eta-x_1), \  \qquad I_-=(x_1-\eta,x_1-\varphi(-y'))
\end{equation*}
and the functions
\begin{align*}
&{\bf \tilde{I}}(y) = \chi_{_{I_+}}(y_1)|x_1+y_1-\varphi(y')|^s\ln\frac1{x_1+ y_1-\varphi(y')}
\\&\qquad \qquad  \qquad  \qquad \qquad  \quad +\chi_{_{I_-}}(y_1)|x_1-y_1-\varphi(-y')|^s\ln\frac1{x_1- y_1-\varphi(y')}  -2x_1^s\ln\frac1{x_1},
\\[1mm]
%\tilde I(y)&=&\chi_{(x_1-\eta,x_1)}(y_1)|x_1-y_1-\varphi(-y')|^{\tau}+\chi_{(-x_1,\eta-x_1)}(y_1)|x_1+y_1-\varphi(y')|^
%{\tau}-2x_1^\tau
%\label{I2}\\
% \mbox{and}\quad  & &\nonumber \\
&{\bf \tilde{J}}(y_1)= \chi_{_{(x_1-\eta,x_1)}}(y_1)|x_1-y_1|^s\ln\frac1{x_1-y_1 }+\chi_{_{(-x_1,\eta-x_1)}}(y_1)|x_1+y_1|^s\ln\frac1{x_1+y_1 }-2x_1^s\ln\frac1{x_1}, \\[1mm]
&{\bf \tilde{I}}_1(y)= \big( \chi_{_{I_+}}(y_1) -\chi_{_{(-x_1,\eta-x_1)}}(y_1)\big) |x_1+y_1|^s\ln\frac1{x_1+y_1},
\\[1mm]
&{\bf \tilde{I}}_{-1}(y)= \big( \chi_{_{I_-}}(y_1) -\chi_{_{(x_1-\eta, x_1)}}(y_1)\big) |x_1-y_1|^s\ln\frac1{x_1-y_1},
\\[1mm]
& {\bf \tilde{I}}_2(y) =
 \chi_{_{I_+}}(y_1) \big(|x_1+y_1-\varphi(y')|^s\ln\frac1{x_1+ y_1-\varphi(y')} -
 |x_1+y_1|^s\ln\frac1{x_1+y_1}\big),
\\[1mm]
& {\bf \tilde{I}}_{-2}(y) =
 \chi_{_{I_-}}(y_1) \big(|x_1-y_1-\varphi(-y')|^s\ln\frac1{x_1- y_1-\varphi(-y')} -
 |x_1-y_1|^s\ln\frac1{x_1-y_1}\big).
 \end{align*}
Observe that
\begin{eqnarray*}
   \cF(x_1)\ge \int_{Q_{\eta}} \frac{{\bf \tilde{I}}(y)}{|y|^{N+2s}}dy-\tilde{E}_0(x_1) = \int_{Q_{\eta}} \frac{{\bf \tilde{J}}(y_1)}{|y|^{N+2s}}dy-\tilde{E}_0(x_1)+\tilde{E}_1(x_1)+\tilde{E}_2(x_1),
\end{eqnarray*}
where
$$
\tilde{E}_0(x_1)=s \, c_{39}\int_{Q_\eta} \frac{\Big(|x_1+y_1-\varphi(y')|^s\ln\frac{1}{|x_1+y_1-\varphi(y')|}+|x_1-y_1-\varphi(y')|^s\ln \frac{1}{|x_1-y_1-\varphi(y')|}\Big)|y'|^2 }{|y|^{N+2s}}dy
$$
 and
\begin{equation*}
\tilde{E}_i(x_1)=\int_{Q_{\eta}} \frac{{\bf \tilde{I}}_i(y)+{\bf \tilde{I}}_{-i}(y)}{|y|^{N+2s}}dy,\quad i=1,2.
\end{equation*}

 Changing variables,  we have that
$$
\int_{Q_{\eta}}\frac{{\bf \tilde{J}}(y_1)}{|y|^{N+2s}}dy
=x_1^{-s}\int_{Q_{\frac{\eta}{x_1}}}    \frac{{\bf \tilde{J} }(x_1z_1)x_1^{-s}}{|z|^{N+2s}}dz= 2x_1^{-s}({\bf \tilde{R}}_1-{\bf\tilde{R} }_2),
$$
where
$$
{\bf \tilde{R}}_1=\int_{Q_{\frac{\eta}{x_1}}^+}\frac{\chi_{_{(0,1)}}(z_1)|1-z_1|^s\ln\frac1{1-z_1}
+(1+z_1)^{s}\ln\frac1{1+z_1} }{|z|^{N+2s}}dz\to -c_1\psi_s'(s)\quad {\rm as}\ x_1\to 0^+
$$
and
$$
{\bf \tilde{R}}_2=\int_{Q_{\frac{\eta}{x_1}}^+}
\frac{\chi_{_{(\frac{\eta}{x_1}-1,\frac{\eta}{x_1})}}(z_1)(1+z_1)^s\ln \frac1{1+z_1}}{|z|^{N+2s}}dz.
$$

Similar to the proof of Step 1, we have that
 \begin{equation*}
\cF(x_1)\ge c_{40}\, x_1^{- s} \Big\{-\psi_s'(s)-c_{41}\, x_1^{s}\ln \frac{1}{x_1} \Big\}
\end{equation*}
and by Step 2, it yields that
 \begin{equation*}
\cF(x_1)\le c_{42}\, x_1^{- s} \Big\{ -\psi_s'(s)+c_{43}\,x_1^{s}\ln \frac{1}{x_1} \Big\}.
\end{equation*}
Therefore,
 $$
  \frac{-\psi_s'(s)}{2} x_1^{-s}\leq  -(-\Delta)^s W_s(x)\leq  -2\psi_s'(s) x_1^{-s}.
$$
This ends the proof.  $\hfill\Box$\medskip

\smallskip

Denote by $G_{s,\Omega}$ the fractional Green kernel of $(-\Delta)^s$ in $\Omega\times\Omega$ and by $\mathbb{G}_{s,\Omega} [\cdot]$ the
fractional Green operator defined as
$$\mathbb{G}_{s,\Omega}[f](x)=\int_{\Omega} G_{s,\Omega}(x,y)f(y)dy,\qquad \forall f\in
{L^1(\Omega,\rho^{s}dx)}. $$
The following estimate is important for  study the existence of   minimal solution to problem (\ref{eq 1.1}).

\begin{lemma}\label{lm 2.1}
Let $\tau\in(0,2s)$, $A_{\frac12}=\{x\in\Omega: \ \rho(x)<\frac12\}$.  For $x\in A_{\frac12}$, we denote
\begin{equation}\label{varrho 1}
\varrho_\tau(x)=\left\{ \arraycolsep=1pt
\begin{array}{lll}
 \rho(x)^{\min\{s,\tau\}}\ \ &{\rm if}\ \ \tau\in (0,s)\cup(s,2s),
 \\[2mm]
\rho(x)^s \ln \frac1{\rho(x)} \ & {\rm if}\ \ \tau=s
\end{array}
\right.
\end{equation}
and we make $C^1$ extension of $\varrho_\tau$ into  $\Omega\setminus A_{\frac12}$ such that
$\varrho_\tau>0$ in $\Omega\setminus A_{\frac12}$.

Then there exists $c_\tau>1$ such that
$$\frac1c_\tau\varrho_\tau(x) \le \mathbb{G}_{s,\Omega}[\rho^{\tau-2s}](x)\le c_\tau\varrho_\tau(x),\quad\forall x\in  \Omega.$$
\end{lemma}

\noindent{\bf Proof.}
For $\tau\in(0,s)$, by Proposition \ref{pr 2.1.0}$(i)$,  there exists $\delta_1>0$ small such that
$$
\frac1 {c_{1}} \rho(x)^{\tau-2s }\leq (-\Delta)^s V_\tau(x)\leq
c_{1}\rho(x)^{\tau-2s}, \quad  \ \ \forall \, x\in A_{\delta_1},$$
for some $c_{1}>0$ depends on $\tau$. Combining with
$(-\Delta)^s\mathbb{G}_{s,\Omega}[\rho^{\tau-2s}]=\rho^{\tau-2s}$ in $\Omega$,
we have that
$$\frac1{ c_{1} }(-\Delta)^s V_\tau(x)\le(-\Delta)^s\mathbb{G}_{s,\Omega}[\rho^{\tau-2s}](x)\le  c_{1} (-\Delta)^s V_\tau(x), \quad  \ \ \forall \, x\in A_{\delta_1}.
$$

When $x\in\partial\Omega$, we have that $V_\tau(x)=0$ and $G_{s,\Omega}(x,y)=0$ for any $y\in\Omega$, then
$$\mathbb{G}_{s,\Omega}[\rho^{\tau-2s}](x)=\int_{\Omega} G_{s,\Omega}(x,y)\rho(y)^{\tau-2s}dy=0.$$
Fixed $0<\delta_2<\delta_1$, for $x\in \Omega$ satisfying $\rho(x)=\delta_2$, we have that $V_\tau(x)=\delta_2^\tau$ and then
  \begin{equation}\label{2.1.3251}
\frac1{c_{44}}V_{\tau}(x)\le \mathbb{G}_{s,\Omega}[\rho^{\tau-2s}](x)\le c_{44}\,V_{\tau}(x).
\end{equation}
By Comparison Principle,   there exists $c_{45}>1$ depends on $\tau$ such that
\begin{equation}\label{2.1.3252}
\frac1{c_{45}}V_{\tau}\le \mathbb{G}_{s,\Omega}[\rho^{\tau-2s}]\le c_{45}\,V_{ \tau}\quad{\rm in}\ \  A_{\delta_2}.
\end{equation}
Since $\mathbb{G}_{s,\Omega}[\rho^{\tau-2s}]$ and $V_{\tau}$ is bounded in $\Omega\setminus A_{\delta_2}$, then (\ref{2.1.3252})
holds in $\Omega$.

\smallskip

For $\tau\in(s,2s)$, by Proposition \ref{pr 2.1.0}$(ii)$, there exists $\delta_1>0$ small such that
$$
\frac1{c_1} \rho(x)^{\tau-2s }\leq -(-\Delta)^s V_\tau(x)\leq
c_1\,\rho(x)^{\tau-2s}, \ \ \ \forall \, x\in A_{\delta_1},$$
for some $c_1>0$ depends on $\tau$.
Denote $W_\tau=\mathbb{G}_{s,\Omega}[1]-V_\tau$, by direct computation,
there exists $\delta_2<\delta_1$ such that
$$\frac{1}{c_{46}}\rho(x)^{\tau-2s}\le (-\Delta)^s W_\tau(x)=-1-(-\Delta)^s V_\tau(x)\le c_{46}\,\rho(x)^{\tau-2s},\quad \forall \ x\in A_{\delta_2}.$$
Since
$(-\Delta)^s\mathbb{G}_{s,\Omega}[\rho^{\tau-2s}]=\rho^{\tau-2s}$ in $\Omega$, then
\begin{equation}\label{2.1.3253}
\frac1{c_{46}}W_{\tau}(x)\le \mathbb{G}_{s,\Omega}[\rho^{\tau-2s}](x)\le c_{46}\,W_{ \tau}(x), \ \ \ \ \forall \ x\in\partial A_{\delta_2}.
\end{equation}
By the Comparison Principle, we have that
$$\frac1{c_{46}}\rho(x)^s\le \frac1{c_{46}}W_\tau(x) \le \mathbb{G}_{s,\Omega}[\rho^{\tau-2s}](x)\le c_{46}\,W_\tau(x) \le c_{46}\,\rho(x)^s, \quad \  x\in A_{\delta_2}$$
%Since $\mathbb{G}_\Omega[\rho^{\tau-2}]$ and $V_{\tau}$ is bounded in $\Omega\setminus A_{\delta_2}$, then
and then
$$\frac1{c_{46}}\rho(x)^s\le \mathbb{G}_{s,\Omega}[\rho^{\tau-2s}](x)\le c_{46}\,\rho(x)^s, \quad \  x\in \Omega. $$

For $\tau=s$, by Proposition  \ref{pr 2.1.0}$(iv)$,  there exists $\delta_1>0$ small such that
 $$
\frac1{c_1}\rho(x)^{-s}\leq-(-\Delta)^s W_s(x)\leq
c_1\,\rho(x)^{-s}, \ \ \ \forall \, x\in A_{\delta_1},$$
then it follows  by Comparison Principle that
$$\frac1{c_{47}} \rho^s\ln \frac1\rho \le \mathbb{G}_{s,\Omega}[\rho^{-s}]\le c_{47} \,\rho^s\ln \frac1\rho \qquad{\rm in }\quad  A_{\delta_2}$$
for some $0<\delta_2<\delta_1$ and then it holds in $\Omega$.
The proof ends.
\qquad $\hfill\Box$

\smallskip

%By Lemma \ref{lm 2.1}, we have the following results.
\begin{corollary}\label{cr 2.1}
 For $\gamma\in(\frac23s,s)$, it holds that
\begin{equation}\label{2.1.3}
 \lim_{x\in\Omega, x\to\partial\Omega}\mathbb{G}_{s,\Omega}[\rho^{-2\gamma}](x)\rho(x)^{-\gamma}=+\infty.
\end{equation}
\end{corollary}
\noindent{\bf Proof.} By Lemma \ref{lm 2.1}, we have that
$$\mathbb{G}_{s,\Omega}[\rho^{-2\gamma}](x)\ge c_{48}\, \rho(x)^{2s-2\gamma},\quad \forall\  x\in\Omega,$$
 combining with the fact that
$2s-2\gamma<\gamma<s$, since $\gamma\in(\frac23s,s)$,  then (\ref{2.1.3}) holds.
\qquad $\hfill\Box$

\section{Minimal solutions}

This section is devoted to the existence of pull-in voltage $\lambda^*$ to problem (\ref{eq 1.1}) such that
 (\ref{eq 1.1}) admits a solution for $\lambda\in(0,\lambda^*)$ and no solution for $\lambda>\lambda^*$.

\begin{proposition}\label{pr 2.1}
 Assume that $a\in C^\gamma(\Omega)\cap C(\bar\Omega)$ satisfies (\ref{1.1}) with $\gamma\in(0,\frac23s]$, then there exists  $\lambda^* >0$ such that
 problem  (\ref{eq 1.1}) admits at least one solution for  $\lambda\in(0,\lambda^*)$ and no solution
for $\lambda>\lambda^*$.
Moreover,
\begin{equation}\label{2.4}
\lambda^*\le \frac{\int_\Omega a(x) dx}{ \int_\Omega\frac{\mathbb{G}_{s,\Omega}[1](x)}{a(x)^2}dx}.
\end{equation}

\end{proposition}

\noindent{\bf Proof.}   Let $v_0\equiv0$ in $\bar\Omega$ and
$$v_1=\lambda \mathbb{G}_{s,\Omega}[a^{-2}]>0,$$
by (\ref{1.1}) and Lemma \ref{lm 2.1}, we have that
\begin{equation}\label{2.428}
 v_1  = \lambda \mathbb{G}_{s,\Omega}[a^{-2}] \le  \frac{\lambda}{\kappa^2} \mathbb{G}_{s,\Omega}[\rho^{-2\gamma}]  \le
 \frac{ \lambda}{\kappa^2}c_{49} \, \varrho_{ 2s-2\gamma },
\end{equation}
where $c_{49}>0$ depending on $\gamma$ and $\varrho_{2s-2\gamma}$ is given by (\ref{varrho 1}).

For $\gamma\not=\frac12s$ and $0<\gamma\le \frac23s$, we observe that $\min\{s,2s-2\gamma\}\ge \gamma$, then by (\ref{2.428}), it yields that
$$
v_1(x)\le  \frac{ \lambda}{\kappa^2}c_{49} \, \rho(x)^{\min\{s,2s-2\gamma\}}\le \frac{\lambda}{\kappa^2}c_{49}\, \rho(x)^\gamma,\qquad \forall \ x\in\Omega.
$$
While for $\gamma=\frac12s$, we have that $2s-2\gamma=s$ and by (\ref{2.428}), the following is true
$$
v_1(x)\le  \frac{ \lambda}{\kappa^2}c_{49}\,  \rho^s(x)\ln\frac1{\rho(x)}\le \frac{\lambda}{\kappa^2}c_{49}\, \rho(x)^\gamma,\qquad \forall \ x\in\Omega.
$$
Then for $\gamma\in (0,\frac23s]$, it holds that
$$v_1(x)\le \frac{\lambda}{\kappa^2}c_{49} \, \rho(x)^\gamma,\qquad \forall \ x\in\Omega.$$
Fixed $\mu\in(0, \kappa)$, then there exists $\lambda_1>0$ such that
%\begin{equation}\label{2.1}
$$\frac{\lambda}{\kappa^2}c_{49}\le \mu<\kappa, \ \ \ \forall \ \lambda<\lambda_1,$$
thus  for any $\lambda<\lambda_1$,
$$v_1(x)\le \mu\rho(x)^\gamma,\qquad \forall \  x\in\Omega.$$
%\end{equation}

Let $v_2=\lambda \mathbb{G}_{s,\Omega}[(a-v_1)^{-2}]$, by the fact that
$a(x)\ge \kappa\rho(x)^\gamma>\mu\rho(x)^\gamma\ge v_1(x)>0$ for any $x\in\Omega$ and Lemma \ref{lm 2.1},
 we have that
$$v_1=\lambda\mathbb{G}_{s,\Omega}[\frac1{a^2}]\le v_2\le  \frac{\lambda}{(\kappa-\mu)^2}\mathbb{G}_{s,\Omega}[\rho^{-2\gamma}]
\le   \frac{\lambda}{(\kappa-\mu)^2}c_{50} \, \varrho_{ 2s-2\gamma }\le  \frac{\lambda}{(\kappa-\mu)^2}c_{50}\, \rho^\gamma\quad {\rm in}\ \ \Omega$$
for $\gamma\in(0,\frac23s]$. Note that there exists $\lambda_2\in(0,\lambda_1]$ such that
$$\frac{\lambda}{(\kappa-\mu)^2}c_{50}\le \mu<\kappa, \ \ \ \forall \ \lambda<\lambda_2,$$
then
$$v_2(x)\le \mu\rho(x)^\gamma,\qquad \forall \ x\in\Omega.$$
Iterating the above process, it holds that
 $$v_n:=\lambda\mathbb{G}_{s,\Omega}[\frac{1}{(a-v_{n-1})^2}]\ge v_{n-1},\quad n\in\N$$
and
$$
v_n(x)\le \mu\rho(x)^\gamma,\qquad\forall \  x\in\Omega.
$$
Then the sequence $\{v_n\}_n$ converges to a limit, denoting by $u_\lambda=\lim\limits_{n\to\infty} v_n$,  and
then $u_\lambda$ is  a classical solution of (\ref{eq 1.1}).

\smallskip

We now claim  that $u_\lambda$ is the minimal solution of  (\ref{eq 1.1}), that is,
for any positive solution $u$ of (\ref{eq 1.1}), it is true that
$u_\lambda\le u$. Indeed, since $u\ge v_0\equiv0$  and
$$u=\lambda \mathbb{G}_{s,\Omega}[\frac1{(a-u)^2}]\ge\lambda\mathbb{G}_{s,\Omega}[\frac1{a^2}]=v_1,$$
by inductively, it holds that $u\ge v_n$ for all $n\in\N$, then $u\ge u_\lambda$.

\smallskip

Next wee show that the mapping $\lambda\mapsto u_\lambda$ is increasing.
If problem  (\ref{eq 1.1}) has a  super solution $u$  for $\lambda_0>0$,
then (\ref{eq 1.1}) admits a minimal solution $u_\lambda$   for all $\lambda\in(0,\lambda_0]$, let us define
   $$\lambda^*=\sup\{\lambda>0:\,  (\ref{eq 1.1}) \ {\rm has \  a \  minimal \  solution \ for} \ \lambda \},$$
which is the largest $\lambda$ such that problem (\ref{eq 1.1}) has minimal positive solution. %Note that $\lambda^*>0$.
In fact, for $0<\lambda_1<\lambda_2<\lambda^*$, we have that $0\le u_{\lambda_1}\le u_{\lambda_2}\le a$ in $\Omega$, then
\begin{align*}
 (-\Delta)^s (u_{\lambda_2}-u_{\lambda_1}) = \frac{\lambda_2}{(a-u_{\lambda_2} )^2} - \frac{\lambda_1}{(a-u_{\lambda_1} )^2}
    \ge  \frac{\lambda_2-\lambda_1}{(a-u_{\lambda_1} )^2}\ge \frac{\lambda_2-\lambda_1}{a^2}>0,
\end{align*}
which implies that
 \begin{equation}\label{e 1.1}
 u_{\lambda_2}-u_{\lambda_1}\ge (\lambda_2-\lambda_1)\mathbb{G}_{s,\Omega}[a^{-2}]>0.
 \end{equation}
So the mapping $\lambda\mapsto u_\lambda$ is increasing.
%, so $u_{\lambda}<a$ in $\Omega$ for any $\lambda<\lambda^*$ and
%by the interior regularity, we have that $u_\lambda\in C_{loc}^{2,\gamma'}(\Omega)$ for any $\gamma'<\gamma$.

Finally, we prove that $\lambda^*<+\infty$.  By contradiction, if not, then for any $\lambda>0$, (\ref{eq 1.1}) has the minimal solution $u_\lambda$.
Let $A_{\delta}=\{x\in\Omega:\, \rho(x)<\delta\}$, for $n\in\N$,
\begin{equation}\label{101}
\eta_n=1 \ \ {\rm in} \  \Omega\setminus{A_{1/n}}, \quad  \   \eta_n=0 \ \ {\rm in} \  A_{1/{2n}},\quad  \  \eta_n\in C^2(\Omega),
\end{equation}
and
$\xi_n= \mathbb{G}_{s,\Omega}[1]\eta_n$ in $\Omega$.
Observe that $\xi_n\in C_c^2(\Omega)$ and
\begin{align*}
(-\Delta)^s \xi_n&=[(-\Delta)^s \mathbb{G}_{s,\Omega}[1]]\cdot\eta_n+
\mathbb{G}_{s,\Omega}[1]\cdot[(-\Delta)^s\eta_n]\\
 &\qquad\qquad\qquad\qquad\quad\ \ +\int_ {\R^N} \frac{[\mathbb{G}_{s,\Omega}[1] (x)-\mathbb{G}_{s,\Omega}[1] (y)]\cdot[\eta_n(x)-\eta_n(y)]} {|x-y|^{N+2s}} dx,
\end{align*}
where $\int_ {\R^N} \frac{[\mathbb{G}_{s,\Omega}[1] (x)-\mathbb{G}_{s,\Omega}[1] (y)]\cdot[\eta_n(x)-\eta_n(y)]} {|x-y|^{N+2s}} ]dx$ is bounded,
% because of H\"{o}lder continuity,by the fact that
$(-\Delta)^s \mathbb{G}_{s,\Omega}[1]=1$, $|(-\Delta)^s\eta_n|\le c_{51} \, n^{2s}$ and
$\mathbb{G}_{s,\Omega}[1]\le c_{52} \, \rho^s$ in $\Omega$, then
\begin{equation}\label{eq9281}
\int_\Omega u_\lambda [(-\Delta)^s\xi_n] dx\le\int_\Omega (1+c_{53})u_\lambda dx+ c_{54} \, n^{2s}\int_{A_{1/n}}u_\lambda \rho^s dx
\le c_{55},
\end{equation}
where the constants are independent on $n$.

Since $u_\lambda$ is the minimal solution of  (\ref{eq 1.1}), we have that
$$\int_\Omega u_\lambda [(-\Delta)^s\xi_n] dx
=\int_\Omega [(-\Delta)^s u_\lambda]\xi_n dx=\int_\Omega \frac{\lambda \xi_n}{(a-u_\lambda)^2} dx.$$
Passing to the limit of $n\to\infty$ and combining with (\ref{eq9281}), we obtain that
$\int_\Omega \frac{\lambda \mathbb{G}_{s,\Omega}[1]}{(a-u_\lambda)^2} dx\le c_{55}$ and then
\begin{align*}
\int_\Omega a(x) dx\ge \int_\Omega u_\lambda(x) dx
&= \int_\Omega [(-\Delta)^s\mathbb{G}_{s,\Omega}[1](x) ] \cdot u_\lambda (x) dx
\\&= \int_\Omega \mathbb{G}_{s,\Omega}[1](x)\cdot [(-\Delta)^s u_\lambda (x)]dx\\
 &=  \lambda \int_\Omega \frac{\mathbb{G}_{s,\Omega}[1](x)}{[a(x)-u_\lambda(x)]^2}dx
 \ge   \lambda \int_\Omega\frac{\mathbb{G}_{s,\Omega}[1](x)}{a^2(x)}dx,
\end{align*}
which implies that
$$
\lambda\le  \frac{\int_\Omega a(x) dx }{\int_\Omega\frac{\mathbb{G}_{s,\Omega}[1](x)}{a^2(x)}dx}.
$$
Then (\ref{2.4}) holds.

\smallskip

Similar to the proof as the claim, we have that  problem  (\ref{eq 1.1}) has the minimal solution for  $\lambda\in(0,\lambda^*)$ and
no solution for $\lambda>\lambda^*$. The proof ends.$\hfill\Box$

\medskip

\noindent{\bf Proof of Theorem \ref{teo 2}.}
By contradiction, assume that there exists  $\lambda>0$ such that
problem (\ref{eq 1.1}) has a solution $u_\lambda$ satisfying $0<u_\lambda<a$ in $\Omega$, then
\begin{equation}\label{2.5}
 \lambda \mathbb{G}_{s,\Omega}[a^{-2}]\le \lambda \mathbb{G}_{s,\Omega}[(a-u_\lambda)^{-2}] = u_\lambda< a\quad{\rm in}\quad \Omega.
\end{equation}
Since $a\in  C(\bar\Omega)$ satisfies that  $0<a \leq c \, \rho ^\gamma$ with $\gamma\in(\frac23s,s)$ and $c>0$, we have that
$$  \lambda \mathbb{G}_{s,\Omega}[a^{-2}]\ge \frac{\lambda}{c^2}\mathbb{G}_{s,\Omega}[\rho^{-2\gamma}].$$
For $\gamma\in(\frac23s,s)$, by Corollary \ref{cr 2.1}, we have that
$$\lim_{x\in\Omega,  x\to\partial\Omega} \mathbb{G}_{s,\Omega}[\rho^{-2\gamma}](x)\rho(x)^{-\gamma}=+\infty,$$
combining with (\ref{2.5}), it yields that
$$\lim_{x\in\Omega,  x\to\partial\Omega} a(x)\rho(x)^{-\gamma}=+\infty,$$
which contradicts (\ref{a 1.0}). Therefore, problem (\ref{eq 1.1}) has no solution under the assumptions of Theorem \ref{teo 2}.
The  proof ends.\qquad$\hfill\Box$

\smallskip

In order to do the boundary decay estimate for $u_\lambda$,  we introduce the following lemma.

\begin{lemma} \label{lm 2.2}

Assume that the function $a\in C^\gamma(\Omega)\cap C(\bar\Omega)$  satisfies (\ref{1.1}) with  $\gamma\in(0,\frac23s]$ and
 $u$ is a super solution of (\ref{eq 1.1})
with $\lambda>0$ such that
\begin{equation}\label{2.3}
u \le \theta a\quad {\rm in}\quad \Omega
\end{equation}
for  some $\theta\in(0,1)$. Then (\ref{eq 1.1}) admits the minimal solution $u_\lambda$  such that
\begin{equation}\label{2.311}
u_\lambda\le c \, \varrho_{ 2s-2\gamma }\quad{\rm in}\quad \Omega
\end{equation}
 for some $c>0$, where $\varrho_{ 2s-2\gamma }$ is given by (\ref{varrho 1}).

\end{lemma}

 \noindent{\bf Proof.}
 Since $u$ is a super solution of (\ref{eq 1.1}) and satisfies (\ref{2.3}),
 similar to the proof of Proposition \ref{pr 2.1},  we have that
  (\ref{eq 1.1}) has the minimal solution $u_\lambda$ and then
$$   u_\lambda\le u\le \theta a,\qquad$$
 by the fact of  (\ref{1.1}), we deduce that
 $$ u_\lambda =\lambda \mathbb{G}_{s,\Omega}[(a-u_\lambda)^{-2}]
   \le \frac{\lambda}{\kappa^{2}(1-\theta)^{2}}\mathbb{G}_{s,\Omega}[\rho^{-2\gamma}].
$$
Using  Lemma \ref{lm 2.1} with $\tau=2s-2\gamma$, we obtain that (\ref{2.311}) holds.
The proof is complete. \qquad $\hfill\Box$

\begin{proposition}\label{pr 3.1.1}
Assume that  $a\in C^\gamma(\Omega)\cap C(\bar\Omega)$ satisfies (\ref{1.1}) and(\ref{a 1.0}) with
$\gamma\in(0,\frac23s]$.
Then  for $\lambda\in(0,\lambda^*)$, there exists $\bar c_1 \ge1$ depending on $\gamma$ such that
$$\frac{\lambda}{\bar c_1}\varrho_{ 2s-2\gamma }(x)\le u_\lambda(x)\le \bar c_1 \lambda\rho(x)^{\gamma},\quad \forall x\in \Omega.$$
Furthermore,  there exists $\lambda_*\le \lambda^*$ such that for $\lambda\in(0,\lambda_*)$,
\begin{equation}\label{2.1.2}
 \frac{\lambda}{\bar c_1 }\varrho_{ 2s-2\gamma }(x)\le u_\lambda(x)\le \bar c_1 \lambda\varrho_{ 2s-2\gamma }(x),\quad \forall x\in \Omega
\end{equation}
 where $\varrho_{ 2s-2\gamma }$ is given by (\ref{lm 2.1}).
\end{proposition}
\noindent{\bf Proof.} {\bf Lower bound.}
By Proposition \ref{pr 2.1}, problem  (\ref{eq 1.1}) admits the minimal solution $u_\lambda$ for $\lambda\in(0,\lambda^*)$,
which is approximated
by an increasing sequence $\{v_n\}_n$,
$$ v_0=0 \quad{\rm and}\quad  v_n=\lambda\mathbb{G}_{s,\Omega}[(a-v_{n-1})^{-2}].$$
 By (\ref{a 1.0}) and Lemma \ref{lm 2.1} with $\tau=2s-2\gamma$, we have that
 \begin{align*}
 u_\lambda\ge v_1=\lambda\mathbb{G}_{s,\Omega}[a^{-2}]
 \ge  \lambda c^2 \mathbb{G}_{s,\Omega}[\rho^{-2\gamma}]
    \ge   c_{\gamma} \lambda\varrho_{ 2s-2\gamma }\quad{\rm in}\quad \Omega
 \end{align*}
for some $c_{\gamma}\in(0,1)$ depending on $\gamma$.

\smallskip

{\bf Upper bound.}  By the proof of Proposition \ref{pr 2.1}, for $\lambda>0$ small and some $\mu\in(0,\kappa)$,
we have that $u_\lambda(x)\le \mu\rho(x)^\gamma$ for $x\in\Omega$,
 then there exists $\theta\in(0,1)$ such that
$$u_{\lambda}(x)\le \theta a(x),\quad x\in\Omega.$$
By Lemma \ref{lm 2.2}, we have that
\begin{equation}\label{eq921}
 u_\lambda\le c \varrho_{ 2s-2\gamma }\quad{\rm in}\quad \Omega.
 \end{equation}
Let
$$\lambda_*=\sup\{\lambda\in(0,\lambda^*):\ \limsup_{x\in\Omega,x\to\partial\Omega}u_\lambda(x)\rho(x)^{-\gamma}<\kappa \},$$
we observe that $\lambda_*\le \lambda^*$ and  (\ref{eq921}) holds for all $\lambda\in(0,\lambda_*)$.
The proof ends.
\qquad $\hfill\Box$

\setcounter{equation}{0}
\section{Estimates for $\lambda^*$ and $\lambda_*$ when $\Omega=B_1(0)$ }

In this section, we do the estimates for $\lambda^*$ and $\lambda_*$ in the case that $\Omega=B_1(0)$
and
\begin{equation}\label{a 1.1}
a(x)=\kappa(1-|x|^2)^{\gamma}.
\end{equation}
Observe that the function $a$ represents the upper semi-sphere type shape in $\R^N$. We have following monotonicity results.

\begin{proposition}\label{lm 3.1}
Let $\Omega=B_1(0)$, the function
$a$ satisfy (\ref{a 1.1})
with $\kappa>0$, $\gamma\in(0,\frac23s]$ and $0<\lambda<\lambda^*(\kappa,\gamma)$. Then

$(i)$ the mappings: $\gamma\mapsto \lambda^*(\kappa,\gamma)$ and $\gamma\mapsto \lambda_*(\kappa,\gamma)$  are decreasing;

$(ii)$ the mapping: $\kappa\mapsto \lambda^*(\kappa,\gamma)$  and $\kappa\mapsto \lambda_*(\kappa,\gamma)$ are increasing.
\end{proposition}
\noindent{\bf Proof.} Let $0<\gamma_2\le \gamma_1\le \frac23s$, $u_1, u_2$  be the minimal solutions of (\ref{eq 1.1})
with $a(x)=a_1(x)=\kappa (1-|x|^2)^{\gamma_1}$ and $a(x)=a_2(x)=\kappa (1-|x|^2)^{\gamma_2}$, respectively.
Observe that $a_1\le a_2$ in $B_1(0)$
and  for any $\lambda\in(0,\lambda^*(\kappa, \gamma_1))$, we have that
$$(-\Delta)^s u_1=\frac{\lambda}{(a_1-u_1)^2}\ge \frac{\lambda}{(a_2-u_1)^2}\quad {\rm in} \ \ B_1(0),$$
that is, $u_1$ is a super solution of (\ref{eq 1.1}) with  $a=a_2$. Similar to the proof of Proposition \ref{pr 2.1},
problem (\ref{eq 1.1})
with  $a=a_2$ admits the minimal solution for any $\lambda\in(0,\lambda^*(\kappa, \gamma_1))$. By the definition of $\lambda^*(\kappa, \gamma_2)$,
it yields that
$$\lambda^*(\kappa, \gamma_2)\ge \lambda^*(\kappa, \gamma_1).$$
Thus, we obtain that the mapping $\gamma\mapsto \lambda^*(\kappa,\gamma)$ is decreasing.
Similar the way to show that $\gamma\mapsto \lambda_*(\kappa,\gamma)$  is decreasing, $\kappa\mapsto \lambda^*(\kappa,\gamma)$
and $\kappa\mapsto \lambda_*(\kappa,\gamma)$ are increasing. The proof ends. \qquad $\hfill\Box$

\smallskip

Now we are ready to give the estimate for $\lambda^*$.

\begin{proposition}\label{pr 3.2}
Assume that $\Omega=B_1(0)$ and  the function $a$ satisfies (\ref{a 1.1}) with $\kappa>0$, $\gamma\in(0,\frac23s]$. Then
\begin{equation}\label{3.1-1}
 \lambda^*(\kappa,\gamma)\le  \frac{\kappa^3\, \mathcal{B}(\frac12,\gamma+1)}{\bar c_{N,s}\, \mathcal{B}(\frac12,s-2\gamma+1)},
\end{equation}
where $\bar c_{N,s}$ is a positive constant and $\mathcal{B}(\cdot,\cdot)$ is the Beta function.
 \end{proposition}
\noindent{\bf Proof.} By Proposition \ref{pr 2.1}, we have that
$$
\lambda^*(\kappa,\gamma)\le \frac{\int_{B_1(0)} a(x) dx}{ \int_{B_1(0)}\frac{\mathbb{G}_{s,{B_1(0)}}[1](x)}{a(x)^2}dx}.
$$
Since $a$ satisfies (\ref{a 1.1}),  by direct computation, we have that
\begin{align*}
\int_{B_1(0)}a(x)dx=\int_{B_1(0)}\kappa(1-|x|^2)^{\gamma}dx
& =\kappa\, |S^{N-1}| \int_0^1 (1-r^2)^{\gamma}dr
\\&=\frac12 \kappa\, |S^{N-1}| \int_0^1 (1-t)^{\gamma}t^{-\frac12}dt
\\&=\frac12 \kappa\, |S^{N-1}| \,\mathcal{B}(\frac12,\gamma+1),
\end{align*}
where $S^{N-1}$ is the unit sphere in $\R^N$.
Observe that there exists $\bar c_{N,s}>0$ such that
 $$ \mathbb{G}_{s,B_1(0)}[1](x)={ \bar c_{N,s}}(1-|x|^2)^s,\qquad \forall \ x\in B_1(0),$$
and then
 \begin{align*}
 \int_{B_1(0)}\frac{\mathbb{G}_{s,B_1(0)}[1](x)}{a(x)^2}dx= \frac{\bar c_{N,s}}{\kappa^2} \int_{B_1(0)}(1-|x|^2)^{s-2\gamma} dx
 =  \frac{\bar c_{N,s}\,|S^{N-1}|}{2\kappa^2} \mathcal{B}(\frac12,s-2\gamma+1).
 \end{align*}
Therefore, (\ref{3.1-1}) holds.  The proof ends.
\qquad$\hfill\Box$

\smallskip

\begin{proposition}\label{pr 3.1}
Assume that $\Omega=B_1(0)$ and $a$ satisfies (\ref{a 1.1})
with $\kappa>0$ and $\gamma\in(0,\frac{2}{3}s]$.
Then there exists $c_s>0$ such that
\begin{equation}\label{3.3}
 \lambda_*(\kappa,\gamma)\ge \frac{4}{27}{c_s}\,\kappa^3.
\end{equation}
\end{proposition}
\noindent{\bf Proof.}
Let $t\in(0,1)$,
denote by $w_t(|x|)=t\kappa (1-|x|^2)^{\frac{2}{3}s}$,  we have that
\begin{align*}
(-\Delta)^s w_t(|x|) &\ge c_s t\kappa  (1-|x|^2)^{-\frac{4}{3}s},\qquad x\in B_1(0),
\end{align*}
for some $c_s>0$. Since $a$ satisfies (\ref{a 1.1}), it yields that
$$
\frac{\lambda }{(a(x)-w_t(|x|))^2}=  \frac{ \lambda }{(1-t)^{2}\kappa^{2}}  (1-|x|^2)^{-\frac{4}{3}s}.
$$
Then, for
$\lambda  \le \lambda_t:=c_s t(1-t)^2\kappa^3$,
$w_t$ is a super solution of problem (\ref{eq 1.1}). Observe that when $t=\frac13$,  $\max_{t\in(0,1)}{t(1-t)^2}=\frac{4}{27}$.
By Lemma \ref{lm 2.2}, we  obtain that problem (\ref{eq 1.1}) admits the minimal solution $u_{\lambda_{\frac13}}$. Then
 $\lambda_*(\kappa,{\frac23}s)\ge \lambda_{\frac13}=\frac{4}{27}{c_s}\, \kappa^3$.
 By Proposition \ref{lm 3.1}, the mapping $\gamma\mapsto \lambda_*(\kappa,\gamma)$  is decreasing,
then (\ref{3.3}) holds.
The proof ends.
\qquad$\hfill\Box$

\smallskip

\begin{lemma}\label{lm 4.4}
Assume that
$a\in C^\gamma(\Omega)\cap C(\bar\Omega)$  satisfies  (\ref{1.1}) and (\ref{a 1.0})
with $c\ge\kappa>0$, $\gamma\in(0,\frac23s]$, $\lambda_*$ is given in Theorem \ref{teo 1} and $ u_\lambda$ is the minimal solution of (\ref{eq 1.1})
with  $\lambda\in(0,\lambda_*)$.
Then $u_\lambda\in H^s_0(\Omega)$ and  there exists $\tilde c_1>0$   such that
$$ \int_{\Omega} \frac{u_\lambda}{(a-u_{\lambda})^2} dx\le \tilde c_1\lambda \qquad{\rm and}\qquad \int_\Omega |(-\Delta)^\frac s2 u_\lambda|^2 dx\le \tilde c_1\lambda^2.$$

\end{lemma}
\noindent{\bf Proof.} For $\lambda\in(0,\lambda_*)$ and $\gamma\in(0,\frac23s]\setminus{\{\frac12s\}}$, by (\ref{2.1.2}), we have that
\begin{equation}\label{eq 001}
 u_\lambda(x)\le \bar c_1\lambda\rho(x)^{\min\{s,2s-2\gamma\}},\quad \  x\in\Omega.
  \end{equation}
Since $\kappa \rho(x)^\gamma\le a(x)\le c \rho(x)^\gamma$ for $x\in\Omega$,
 there exists $\theta_1\in(0,1)$ such that $u_\lambda<\theta_1 a$ in $A_{\delta}=\{x\in\Omega:\, \rho(x)<\delta\}$ for $\delta>0$ small.
Moreover, by the fact that $u_\lambda<a $ in $\Omega$, there exists $\theta_2\in(0,1)$ such that $u_\lambda\le \theta_2a$ in $\Omega\setminus{A_{\delta}}$.
Let $\theta=\max\{\theta_1,\theta_2\}$,  we have that $u_\lambda\le \theta a$ in $\Omega$ and then
\begin{equation}\label{eq 01}
a-u_\lambda\ge (1-\theta) a \ge (1-\theta)\kappa \rho^\gamma\quad {\rm in} \ \ \Omega.
 \end{equation}
Therefore,   combining with (\ref{eq 001}), (\ref{eq 01}), and by the fact that $\min\{s,2s-2\gamma\}-2\gamma>-s>-1$,  we have that
 \begin{equation}\label{4.1.2}
 \int_\Omega \frac{u_\lambda}{(a-u_\lambda)^2}dx\le \int_\Omega \frac{ \bar c_1\lambda\rho(x)^{\min\{s,2s-2\gamma\}}}{(1-\theta)^2\kappa^2 \rho(x)^{2\gamma}}dx:=\tilde c_1\lambda<+\infty,
 \end{equation}
 Taking a sequence $\{\xi_n\}_n\subset C_c^2(\Omega)$  which converges to
$u_\lambda$ as $n\to\infty$, since $u_\lambda$ is the minimal solution of  (\ref{eq 1.1}), we have that
$$\int_\Omega (-\Delta)^\frac s2 u_\lambda\cdot (-\Delta)^\frac s2{\xi_n} dx=\int_\Omega \frac{\lambda \xi_n}{(a-u_\lambda)^2} dx.$$
Passing to the limit as $n\to\infty$, we obtain that
$$
 \int_\Omega  |(-\Delta)^\frac s2 u_\lambda|^2dx =  \int_\Omega \frac{\lambda u_\lambda}{(a-u_\lambda)^2}dx\le \tilde c_1\lambda^2.
$$

For $\lambda\in(0,\lambda_*)$ and $\gamma=\frac12s$,   by (\ref{2.1.2}), we have that
$$ u_\lambda(x)\le \bar c_1\lambda\rho(x)^s \ln\frac{1}{\rho(x)},\quad \  x\in\Omega,$$
similar to show that there exists $\theta\in(0,1)$ such that $u_\lambda\le \theta a$ in $\Omega$ and then (\ref{eq 01}) holds, then
 \begin{equation}\label{4.1.20p}
 \int_\Omega \frac{u_\lambda}{(a-u_\lambda)^2}dx\le \int_\Omega \frac{ \bar c_1\lambda}{(1-\theta)^2\kappa^2 }\ln\frac{1}{\rho(x)}dx\le \tilde c_1\lambda,
 \end{equation}
therefore, $\int_\Omega  |(-\Delta)^\frac s2 u_\lambda|^2dx \le \tilde c_1\lambda^2.$
The proof ends.\qquad$\hfill\Box$

\medskip

\noindent{\bf Proof of Theorem \ref{teo 1}.} The existence of   minimal solution for $\lambda\in(0,\lambda^*)$
and the nonexistence for $\lambda>\lambda^*$ follow by Proposition \ref{pr 2.1}. The proof of  Theorem \ref{teo 1} $(iii)$
is true by Proposition \ref{pr 3.1.1} and Lemma \ref{lm 4.4}. The estimates of $\lambda^*$ and $\lambda_*$ are obtained by Proposition \ref{lm 3.1},
Proposition \ref{pr 3.2} and Proposition \ref{pr 3.1}.\qquad $\hfill\Box$

\setcounter{equation}{0}
\section{Properties of minimal solution}

%\subsection{Regularity}

In this section, we first study the solutions of (\ref{eq 1.1}) when $\lambda=\lambda^*$.
\begin{proposition}\label{pr 4.1}

Assume that the function
$a\in C^\gamma(\Omega)\cap C(\bar\Omega)$  satisfies (\ref{1.1}) and (\ref{a 1.0})
with $c\ge\kappa>0$, $\gamma\in(0,\frac23s]$ and $ u_{\lambda^*}$ is given by (\ref{4.2}).
Then  $ u_{\lambda^*}$ is a weak solution of (\ref{eq 1.1}) with $\lambda=\lambda^*$. Moreover,
for any $\beta\in(0,\min\{\gamma-s+1,s\})$,  there exists $c_\beta>0$ such that
 \begin{equation}\label{4.0.5}
 \norm{u_{\lambda^*}}_{W^{s,\frac{N}{N-\beta}}(\Omega)}\le c_\beta
 \end{equation}
and
\begin{equation}\label{4.0.6}
 \int_{\Omega} \frac{\rho^{s-\beta}}{(a-u_{\lambda^*})^2} dx\le c_\beta.
 \end{equation}
\end{proposition}
\noindent{\bf Proof.}
 For any  $\beta\in(0,\min\{\gamma-s+1,s\})$ and $n\in\N$, denote $\xi_n=\mathbb{G}_{s,\Omega}[\rho^{-s-\beta}]\eta_n$, where $\eta_n$ is given by (\ref{101}),
 we observe that $\xi_n\in C_c^2(\Omega)$ and
\begin{align*}
(-\Delta)^s\xi_n&=   \rho^{-s-\beta}\eta_n+ \mathbb{G}_{s,\Omega}[\rho^{-s-\beta}]\cdot[(-\Delta)^s\eta_n]\\[2mm]
&\quad + \int_ {R^N} \frac{[\mathbb{G}_{s,\Omega}[\rho^{-s-\beta}] (x)-\mathbb{G}_{s,\Omega}[\rho^{-s-\beta}] (y)]\cdot[\eta_n(x)-\eta_n(y)]} {|x-y|^{N+2s}} dx.
\end{align*}
Using Lemma \ref{lm 2.1} with $\tau=s-\beta\in(\max\{2s-\gamma-1,0\},s)$, we have that $\mathbb{G}_{s,\Omega}[\rho^{-s-\beta}]\le c_{56}\rho^{s-\beta}$ in $\Omega$.
Combining with $|(-\Delta)^s\eta_n|\le c_{51}n^{2s}$, it holds that
$$\int_ {\R^N} \frac{[\mathbb{G}_{s,\Omega}[\rho^{-s-\beta}] (x)-\mathbb{G}_{s,\Omega}[\rho^{-s-\beta}] (y)]\cdot[\eta_n(x)-\eta_n(y)]} {|x-y|^{N+2s}} )dx\leq c_{57}.$$ Since  $0<u_\lambda <a \leq c \rho^\gamma$ in $\Omega$, we have that
\begin{align*}
  \int_\Omega u_\lambda [(-\Delta)^s\xi_n] dx&\le  c\int_\Omega  \rho^{\gamma-s-\beta} dx+c_{56}n^{2s}\int_{A_{1/n}}  \rho^{\gamma-\beta+s}dx
  +  c\, c_{57}\int_\Omega  \rho^{\gamma} dx \le   \bar c_{\beta},
\end{align*}
where $c_{56}, c_{57}>0$ independent on $n$ and $\bar c_\beta>0$ satisfies that $\bar c_\beta\to+\infty$ as $\beta\to\min\{\gamma-s+1,s\}^-$. Thus,
$$\int_\Omega \frac{ \lambda  \mathbb{G}_{s,\Omega}[\rho^{-s-\beta}]}{(a-u_\lambda)^2} dx\le \bar c_\beta$$
and then
\begin{equation}\label{4.6p}
\int_\Omega \frac{ \lambda  \rho^{s-\beta}}{(a-u_\lambda)^2} dx\le \frac{\bar c_\beta}{c_{58}}.
\end{equation}
By Theorem \ref{teo 1},
%we see that the mapping $\lambda\mapsto u_\lambda$ is increasing and uniformly bounded by the function $a$, which is in $L^1(\Omega)$, then
we have that
$$u_\lambda\to u_{\lambda^*}\ \ \  {\rm in}\ \ L^1(\Omega)\quad \ {\rm as}\ \ \lambda\to \lambda^*$$
and then for any $\xi\in C_c^2(\Omega)$, it yields that
\begin{equation}\label{4.6p1}
\int_\Omega u_\lambda[(-\Delta)^s\xi] dx\to \int_\Omega u_{ \lambda^*}[(-\Delta)^s\xi] dx \ \ \ {\rm as} \ \ \lambda\to  \lambda^*.
\end{equation}
Moreover,  the mapping $\lambda\mapsto  \frac{\lambda  }{(a-u_\lambda)^2}$ is increasing and by (\ref{4.6p}), we have that
$$\frac{\lambda}{(a-u_\lambda)^2}\to \frac{\lambda^*  }{(a-u_{\lambda^*})^2}\ \ \  {\rm a.e.\ in }\ \  \Omega\quad {\rm as}\  \ \lambda\to \lambda^*,$$
$$\frac{\lambda}{(a-u_\lambda)^2}\to \frac{\lambda^*  }{(a-u_{\lambda^*})^2}\ \ \  {\rm in}\ \ L^1(\Omega,\,\rho^{s-\beta}dx)\quad {\rm as}\ \ \lambda\to \lambda^*$$
and
\begin{equation}\label{4.6p3}
\int_\Omega \frac{ \lambda^*  \rho^{s-\beta}}{(a-u_\lambda^*)^2} dx\le \frac{\bar c_\beta}{c_{58}},
\end{equation}
thus, for any $\xi\in C_c^2(\Omega)$, we have that
\begin{equation}\label{4.6p2}
\int_\Omega \frac{\lambda \xi }{(a-u_\lambda)^2} dx\to \int_\Omega \frac{\lambda^* \xi }{(a-u_{\lambda^*})^2} dx \ \ \ {\rm as} \ \ \lambda\to  \lambda^*.
\end{equation}
Since $u_\lambda$ is the minimal solution of  (\ref{eq 1.1}), it yields that
$$
\int_\Omega u_\lambda[(-\Delta)^s\xi] dx=\int_\Omega \frac{\lambda \xi}{(a-u_\lambda)^2} dx,\ \ \ \ \ \forall \, \xi\in C_c^2(\Omega),
$$
 passing to the limit as $\lambda\to\lambda^*$, combining with (\ref{4.6p1}) and (\ref{4.6p2}),
then
 $u_{\lambda^*}$ is a weak solution of (\ref{eq 1.1}) with $\lambda=\lambda^*$.

We recall  $G_{s,\Omega}$ the Green kernel of $(-\Delta)^s$ in
$\Omega\times\Omega $ and by $\mathbb{G}_{s,\Omega}[\cdot]$ the Green operator
defined by
\begin{equation*}\label{optimal0}
\mathbb{G}_{s,\Omega}[f](x)=\int_{\Omega}G_{s,\Omega}(x,y)f(y) dy,\qquad\forall f\in
L^1(\Omega,\rho^{s} dx).
\end{equation*}
It is known from \cite[Proposition 2.5]{CV1} that for $t_1\in[0,s]$ and $p\in\big(1,\frac{N}{N+t_1-2s})$, there is $c_{p,t_1}>0$ such that
\begin{equation}\label{einq-1111}
 \|\mathbb{G}_{s,\Omega}[f]\|_{W^{2s-t_2,p}(\Omega)} \leq  c_{p,t_1} \norm{f}_{L^1(\Omega, \, \rho^{t_1}dx)},
\end{equation}
where
 $$t_2=t_1+\frac{N(p-1)}{p}\ \ \ \  {\rm if}\ t_1>0,  \qquad   t_2>\frac{N(p-1)}{p}\ \ \ \  {\rm if}\ t_1=0.$$
Taking $t_1=s-\beta$,  $t_2=s$,  $p=\frac{N}{N-\beta}$ and $f=(a-u_{\lambda^*})^{-2}$ in (\ref{einq-1111}),
we obtain that
$$
\norm{ u_{\lambda^*}}_{W^{s,\frac{N}{N-\beta}}(\Omega)} \le
c_{p,t_1}\norm{(a-u_{\lambda^*})^{-2}}_{L^1(\Omega, \, \rho^{s-\beta}dx)}\le \frac{c_{p,t_1} \, \bar c_\beta}{c_{58}\, \lambda^*}
:= c_\beta,
$$
where the last inequality used (\ref{4.6p3}).
Thus, (\ref{4.0.5}) holds and  (\ref{4.0.6}) is obvious by (\ref{4.6p3}). \qquad$\hfill\Box$\medskip

%\subsection{Stability}

Now we are position to introduce the stability of minimal solution $u_\lambda$ for problem (\ref{eq 1.1}).
By the definition of $\lambda_*$, we observe that for
any $\lambda\in (0,\lambda_*)$, there exists $\theta\in(0,1)$ such that $u_\lambda\le \theta  a$.
Since $\gamma\in(0,\frac23s]$ and $a\ge \kappa \rho^\gamma$ in $\Omega$,  we have that
\begin{equation}\label{5.1}
\frac{1}{(a-u_\lambda)^3}\le \frac{1}{(1-\theta)^3a^3}\le (1-\theta)^{-3}\kappa^{-3}\rho^{-3\gamma}\le   c_{60}\rho^{-2s}\quad{\rm in}\quad \Omega,
\end{equation}
where $c_{60}>0$ depends on $\theta, \kappa, \gamma$.
This enables us to consider the first eigenvalue $\mu_1(\lambda)$ of $(-\Delta)^s-\frac{2\lambda}{(a-u_\lambda)^3}$ in $H_0^s(\Omega)$,
that is,
\begin{align*}
\mu_1(\lambda)=\inf_{\varphi\in H_0^s(\Omega)}\frac{\int_\Omega (|(-\Delta)^\frac s2 \varphi|^2-\frac{2\lambda \varphi^2}{(a-u_\lambda)^3})dx}{\int_\Omega \varphi^2\,dx}
%\\=
%\inf_{\varphi\in H_0^s(\Omega)}\frac{\int_\Omega (\varphi\cdot(-\Delta)^s \varphi-\frac{2\lambda \varphi^2}{(a-u_\lambda)^3})dx}{\int_\Omega \varphi^2\,dx}
\end{align*}
for $\lambda\in(0,\lambda_*)$. Note that  $u_\lambda$ is stable if $\mu_1(\lambda)>0$  and semi-stable if $\mu_1(\lambda)\ge0$.

\begin{lemma}\label{lm 4.1}
Assume that $\lambda\in(0,\lambda_*)$,
$a\in C^\gamma(\Omega)\cap C(\bar\Omega)$ satisfies (\ref{1.1}) and (\ref{a 1.0})
with $c \ge\kappa>0$, $\gamma\in(0,\frac23s]$. Suppose that $u_\lambda$ is the minimal  solution of (\ref{eq 1.1})
and $v_\lambda$ is a super solution of (\ref{eq 1.1}).

If $\mu_1(\lambda)>0$, then
$$u_\lambda\le v_\lambda\quad{\rm in}\quad \Omega.$$

If $\mu_1(\lambda)=0$, then

$$u_\lambda= v_\lambda\quad{\rm in}\quad \Omega.$$
\end{lemma}
\noindent{\bf Proof.} %We observe that $a>u_\lambda$ in $\Omega$ for $\lambda\in(0,\lambda^*)$ and $\mu_1(\lambda)$ is well-defined for $\lambda\in(0,\lambda_*)$, \textcolor {red}
%{then it follows the procedure of the proof of Lemma 4.1 in \cite{GG} just replacing
%$\frac{f}{(1-u)^2}$ by $\frac{1}{(a-u)^2}$ and replacing $-\Delta$ by $(-\Delta)^s$ .}
The procedure of proof is similar as in Lemma 4.1 in \cite{CWZ1} (also see  \cite{GG}).
\hfill$\Box$

\begin{proposition}\label{pr 4.2}
Let $\lambda\in(0,\lambda_*)$,
$a\in C^\gamma(\Omega)\cap C(\bar\Omega)$ satisfying  (\ref{1.1}) and (\ref{a 1.0})
with $c\ge\kappa>0$, $\gamma\in(0,\frac23s]$, $u_\lambda$ be the minimal solution of (\ref{eq 1.1}).

Then $u_\lambda$ is stable.
\end{proposition}

\noindent{\bf Proof.}
 % We first claim  that the mapping $\lambda\mapsto \mu_1(\lambda)$ is locally Lipschitz continuous and strictly decreasing in $(0,\lambda_*)$.
Observe that the mapping $\lambda\mapsto u_{\lambda}$ is strictly increasing and so is $\lambda\mapsto \frac{2\lambda}{(a-u_\lambda)^3}$,
then the mapping $\lambda\mapsto \mu_1(\lambda)$ is strictly decreasing
in $(0,\lambda_*)$. Let $0<\lambda_1<\lambda_2<\lambda_*$ and $\phi_{\lambda_2}$ be the achieved function of $\mu_1(\lambda_2)$ with $\norm{\phi_{\lambda_2}}_{L^2(\Omega)}=1$, we have that
\begin{align*}
0< \mu_1(\lambda_1)-\mu_1(\lambda_2) &\le  \int_\Omega \left(|(-\Delta)^\frac s2 \phi_{\lambda_2}|^2-\frac{2\lambda_1 \phi_{\lambda_2}^2}{(a-u_{\lambda_1})^3}\right)dx     -\int_\Omega \left(|(-\Delta)^\frac s2 \phi_{\lambda_2}|^2-\frac{2\lambda_2 \phi_{\lambda_2}^2}{(a-u_{\lambda_2})^3}\right)dx
   \\&\le  2(\lambda_2-\lambda_1)\int_\Omega\frac{ \phi_{\lambda_2}^2}{(a-u_{\lambda_2})^3} dx,
\end{align*}
then the mapping $\lambda\mapsto \mu_1(\lambda)$ is locally Lipschitz continuous.

Now we show that $u_\lambda$ is stable and $\mu_1(\lambda)>0$ for $\lambda>0$ small.
In fact, by the inequality (1) in \cite{D} (also see (1.10) in \cite{CS3}), there exists constant $c_{61}>0$ such that
\begin{align*}
\int_\Omega \varphi(x)^2\rho(x)^{-2s}dx \le c_{61}\int_\Omega \int_\Omega \frac{|\varphi(x)-\varphi(u)|^2}{|x-y|^{N+2s}}dxdy
& \le  c_{61}\int_{\R^N} \int_{\R^N} \frac{|\varphi(x)-\varphi(u)|^2}{|x-y|^{N+2s}}dxdy
\\&= c_{61}\int_{\Omega} |(-\Delta)^\frac s2 \varphi(x)|^2 dx,\qquad \forall \varphi\in H^s_0(\Omega),
\end{align*}
combining with  (\ref{5.1}),  we have that for $\lambda<\lambda_*$,
$$\int_{\Omega}\frac{2\lambda \varphi^2}{(a-u_\lambda)^3}dx\le 2\lambda c_{60}\int_\Omega \varphi^2\rho^{-2s}dx
\le \lambda c_{61} \int_{\Omega}|(-\Delta)^\frac s2 \varphi|^2  dx.$$
For   $\lambda>0$ small such that $\lambda c_{61}< 1$, we obtain that
$$\int_{\Omega}\frac{2\lambda \varphi^2}{(a-u_\lambda)^3}dx
<  \int_{\Omega} |(-\Delta)^\frac s2 \varphi|^2 dx, \quad \forall \varphi\in H_0^s(\Omega)\setminus\{0\},$$
that is, $\mu_1(\lambda)>0$ and $u_\lambda$ is  stable for $\lambda>0$ small.

By the fact that $\mu_1(\lambda)>0$ for $\lambda>0$ small and
 the mapping $\lambda\mapsto \mu_1(\lambda)$ is locally Lipschitz continuous, strictly decreasing
in $(0,\lambda_*)$, so if there exists $\lambda_0\in(0,\lambda_*)$ such that
$\mu_1(\lambda_0)=0$, then $\mu_1(\lambda)>0$ for $\lambda\in(0,\lambda_0)$.
Letting $\lambda_1\in (\lambda_0,\lambda_*)$, the minimal solution $u_{\lambda_1}$  is
a super solution of
$$
\left\{\arraycolsep=1pt
\begin{array}{lll}
 (-\Delta)^s   u = \frac{\lambda_0 }{(a-u)^2}\quad  &{\rm in}\quad\ \Omega,
 \\[2mm]
 \phantom{-\Delta \quad \ }
 u=0\quad &{\rm in}\quad   \   \Omega^c
\end{array}
\right.
$$
and it infers from Lemma \ref{lm 4.1} that
$$u_{\lambda_1}=u_{\lambda_0},$$
 which contradicts that the mapping $\lambda\mapsto u_{\lambda}$ is strictly increasing.
 Therefore, $\mu_1(\lambda)>0$ for $\lambda\in(0,\lambda_*)$  and then $u_{\lambda}$ is stable.
 \qquad$\hfill\Box$

\medskip

%\subsection{Particular case that $\gamma=\frac23s$}

%In this subsection,
We next study further on properties of the minimal solution  when $\gamma=\frac23s$.

\begin{lemma}\label{lm 5.1}
Assume that
$a\in C^\gamma(\Omega)\cap C(\bar\Omega)$ satisfies  (\ref{1.1}) and (\ref{a 1.0})
with $c \ge\kappa>0$, $\gamma=\frac23s$. Then $\lambda^*=\lambda_*$.

\end{lemma}
\noindent{\bf Proof.} For $\lambda\in(0,\lambda^*)$,
by (\ref{e 1.1}) and Lemma \ref{lm 2.1}, we have that
 \begin{align*}
  a-u_{\lambda}& \ge (\lambda^*-\lambda)\mathbb{G}_{s,\Omega}[a^{-2}]\\&\ge c_{62}(\lambda^*-\lambda)\rho^{\gamma}
    \ge  c_{63}(\lambda^*-\lambda)a,
 \end{align*}
 then there exists $\theta\in(0,1)$ such that $u_\lambda\le \theta a$ in $\Omega$. It follows by Lemma \ref{lm 2.2} and the definition of $\lambda_*$ that $\lambda_*=\lambda^*$.
\qquad$\hfill\Box$

\begin{proposition}\label{pr 4.3}

Assume that
$a\in C^\gamma(\Omega)\cap C(\bar\Omega)$ satisfies  (\ref{1.1}) and (\ref{a 1.0})
with $c \ge\kappa>0$,  $\gamma=\frac23s$ and $ u_{\lambda^*}$ is given by (\ref{4.2}).
Then  $ u_{\lambda^*}$ is a semi-stable weak solution of (\ref{eq 1.1}) with $\lambda=\lambda^*$.

\end{proposition}
\noindent{\bf Proof.} When  $\gamma=\frac23s$,
by Lemma \ref{lm 5.1} and Proposition \ref{pr 4.2}, we have that  $u_\lambda$ is stable for $\lambda\in(0,\lambda^*)$, then
$$
\int_\Omega\frac{2\lambda\varphi^2}{(a-u_\lambda)^3}dx<\int_\Omega \varphi\cdot(-\Delta)^s \varphi dx,\quad \ \forall \varphi\in H^s_0(\Omega)\setminus\{0\},
$$
let $\varphi=\mathbb{G}_{s,\Omega}[1]$, we have that
$$\int_\Omega\frac{ \rho^2}{(a-u_\lambda)^3}dx<c_{64}\lambda^{-1}.$$
Since the mapping $\lambda\mapsto \frac{ \rho^2}{(a-u_\lambda)^3}$ is strictly increasing
and bounded in $L^1(\Omega)$,
then
$$\frac{ \rho^2}{(a-u_\lambda)^3}\to \frac{ \rho^2}{(a-u_{\lambda^*})^3}\quad{\rm as}\quad \lambda\to\lambda^*\quad{\rm in}\quad L^1(\Omega)$$
and for any $\varphi\in C_c^2(\Omega)$,
$$\lim_{\lambda\to \lambda^*}\int_\Omega\frac{\lambda\varphi^2}{(a-u_\lambda)^3}dx=\int_\Omega\frac{\lambda^*\varphi^2}{(a-u_{\lambda^*})^3}dx.$$
Therefore,
$$\int_\Omega\frac{2\lambda^*\varphi^2}{(a-u_{\lambda^*})^3}dx\le \int_\Omega \varphi\cdot(-\Delta)^s \varphi dx,\quad \forall \varphi\in C_c^2(\Omega),$$
by the fact that $C_c^2(\Omega)$ is dense in $H^s_0(\Omega)$,
then $u_{\lambda^*}$ is semi-stable.
\qquad$\hfill\Box$

\smallskip

\begin{lemma}\label{lm 4.2}
Let $\lambda\in(0,\lambda^*)$ and $a(x)=\kappa(1-|x|^2)^{\frac23s}$
with $\kappa>0$. Assume that $u$ is a weak solution of (\ref{eq 1.1}) in $B_1(0)$ satisfying for any compact set $K\subset B_1(0)$, there exists
$\hat c_1>0$ such that
 \begin{equation}\label{4.1.3}
 \norm{\frac1{a-u}}_{L^{\frac{3N}{2s}}(K)}\le \hat c_1.
 \end{equation}
Then there exists $\check c_1>0$ depending on $K$ such that
\begin{equation}\label{4.2.1}
 \inf_{x\in K}(a(x)-u(x))>\check c_1.
\end{equation}
\end{lemma}

\noindent{\bf Proof.}  By (\ref{4.1.3}), we have that
$$\frac1{(a-u)^2}\in L^{\frac{3N}{4s}}(K),$$
then $u\in W^{2s,\frac{3N}{4s}}(K)$.
By \cite[Theorem 8.2]{NPV}, we
deduce that $u\in C^{\frac23s}(K')$ with $K'$  compact set in interior point set of $K$.

If there exists $x_0\in B_1(0)$ such that $u(x_0)=a(x_0)$.
For compact set $K\subset B_1(0)$ containing $x_0$,
  we have that
\begin{align*}
 |a(x)-u(x)|&\le  |a(x)-a(x_0)|+|u(x)-u(x_0)|+|u(x_0)-a(x_0)|\\ &=
|a(x)-a(x_0)|+|u(x)-u(x_0)|\le \hat c_1|x-x_0|^{\frac23s},
\end{align*}
then
$$\int_{K}\frac1{(a-u)^{\frac{3N}{2s}}}dx\ge \check c_1 \int_{K} |x-x_0|^{-\frac{3N}{2s}\cdot\frac{2s}3}dx=+\infty,$$
which  contradicts (\ref{4.1.3}). Therefore, (\ref{4.2.1}) holds.
 \qquad$\hfill\Box$
0

\begin{proposition}\label{pr 4.4}
Let $\Omega=B_1(0)$,  $1\le N\le \frac{14}{3}s$,
 the function
$a$ satisfy (\ref{a 1.1})
with $\kappa>0$ and $\gamma=\frac23s$, $ u_{\lambda^*}$ is given by (\ref{4.2}).
Then  $u_{\lambda^*}$ is a classical solution of (\ref{eq 1.1}) with $\lambda=\lambda^*$.

\end{proposition}
\noindent{\bf Proof.} Since the mapping $\lambda\mapsto u_\lambda$ is strictly increasing and bounded by $a$,
then from (\ref{a 1.0}) and Lemma \ref{lm 4.2},
we only have to improve the regularity of $ u_{\lambda^*}$ in any compact set of $B_1(0)$.
For $\lambda\in(0,\lambda^*)$, we know that $u_\lambda$ is stable, i.e.
\begin{equation}\label{4.3.1}
\int_{B_1(0)}\frac{2\lambda \varphi^2}{(a-u_\lambda)^3}dx< \int_{\Omega} \varphi\cdot(-\Delta)^s \varphi  dx, \quad \forall \varphi\in H^s_0(B_1(0))\setminus\{0\}.
\end{equation}

Now we claim that the minimal solutions $u_{\lambda}$ is radially symmetric for $\lambda\in(0,\lambda^*]$. Indeed,
the minimal solution $u_\lambda$ could be approximated
by the sequence of functions
$$v_n=\lambda\mathbb{G}_{s,B_1(0)}[\frac1{(a-v_{n-1})^2}]\ \ \ {\rm with} \ \  v_0=0.$$
It follows by radially symmetry of $v_{n-1}$ and $a$ that $v_n$ is radially symmetry and then
$u_\lambda$ is radially symmetric. Then $u_{\lambda^*}$ is radially symmetric.

By (\ref{4.0.6}),   there exists a  sequence $\{r_n\}_n\subset(0,1)$ such that
$$\lim_{n\to+\infty} r_n=1\quad{\rm and}\quad  a(r_n)-u_{\lambda^*}(r_n)>0.$$
%Otherwise, there exists $r\in(0,1)$ such that
%$$a-u_{\lambda^*}=0\quad{\rm a.e.\ in}\ \ B_1(0)\setminus{B_r(0)},$$
%which contradicts (\ref{4.0.6}).
Let us denote
$$
\varphi_\theta= \left\{\arraycolsep=1pt
\begin{array}{lll}
(a-u_\lambda)^\theta-\epsilon_\lambda^\theta\quad  &{\rm in}\quad B_{r_n}(0),
 \\[1.5mm]
 \phantom{}
0\quad  &{\rm in}\quad \mathbb{R}^N\setminus B_{r_n}(0),
\end{array}
\right.
$$
for $\theta\in(-2,0)$ and $\epsilon_\lambda= a(r_n)-u_\lambda(r_n)$. We observe that $\varphi_\theta\in H_0^s(B_1(0))$.
It follows by (\ref{4.3.1}) with $\varphi_\theta$, since
\begin{align}
  (-\Delta)^s (a-u_\lambda)^\theta &=\theta (a-u_\lambda)^{(\theta-1)}[(-\Delta)^s (a-u_\lambda)]\nonumber  \\
  &\quad \quad - \frac {\theta(\theta-1)(a-u_\lambda)^{(\theta-2)}}{2}\int_\Omega\frac{[(a-u_\lambda)(x)-(a-u_\lambda)(y)]^2}{|y-x|^{}N+2s}dy
  \nonumber \\
  &\leq \theta (a-u_\lambda)^{(\theta-1)}[(-\Delta)^s (a-u_\lambda)],\nonumber
\end{align}
we have that
\begin{eqnarray}
  \int_{B_{r_n}(0)}\frac{2\lambda [(a-u_\lambda)^\theta-\epsilon_\lambda^\theta]^2}{(a-u_\lambda)^{3}}dx &\le &
 \int_{B_{r_n}(0)}[(-\Delta)^s (a-u_\lambda)^\theta] \cdot (a-u_\lambda)^\theta dx\nonumber \\
   &\leq&  \theta\int_{B_{r_n}(0)}(a-u_\lambda)^{2\theta-1} \cdot [(-\Delta)^s (a-u_\lambda)] dx\\
   &=&  \theta\int_{B_{r_n}(0)} [(-\Delta)^\frac s2 (a-u_\lambda)^{2\theta-1}] [(-\Delta)^\frac s2 (a-u_\lambda)] dx.\label{4.3.2}
\end{eqnarray}
Since $u_\lambda$ is the minimal solution of (\ref{eq 1.1}), then
\begin{equation}\label{4.3.3}
  (-\Delta)^s (a-u_\lambda)=(-\Delta)^s a-\frac{\lambda}{(a-u_\lambda)^2}\quad {\rm in} \ \  B_{r_n}(0).
\end{equation}
Multiplying  (\ref{4.3.3})  by $\theta [(a-u_\lambda)^{2\theta-1}-\epsilon_\lambda^{2\theta-1}]$ and applying integration by parts yields that
\begin{align*}
\quad &\int_{B_{r_n}(0)}[(-\Delta )^s a-\frac{\lambda}{(a-u_\lambda)^2}][(a-u_\lambda)^{2\theta-1}-\epsilon_\lambda^{2\theta-1}]dx
\\=&\int_{B_{r_n}(0)}[(-\Delta )^\frac s2 (a-u_\lambda)]\cdot[(-\Delta )^\frac s2 (a-u_\lambda)^{2\theta-1}]dx,
\end{align*}
 then together with (\ref{4.3.2}), we deduce that
\begin{align*}
\int_{B_{r_n}(0)}\frac{2\lambda [(a-u_\lambda)^\theta-\epsilon_\lambda^\theta]^2}{(a-u_\lambda)^{3}}dx \le   -\theta \int_{B_{r_n}(0)}[ \Delta^s a+\frac{\lambda}{(a-u_\lambda)^2}][(a-u_\lambda)^{2\theta-1}-\epsilon_\lambda^{2\theta-1}]dx,
\end{align*}
where  $\Delta^s=-(-\Delta)^s$,  then
\begin{align*}
\lambda(2+\theta)\int_{B_{r_n}(0)}\frac1{(a-u_\lambda)^{3-2\theta}}dx&\le \int_{B_{r_n}(0)}\frac{4\lambda\epsilon_\lambda^\theta}{(a-u_\lambda)^{3-\theta}}dx
-\int_{B_{r_n}(0)}\frac{2\lambda\epsilon_\lambda^{2\theta}}{(a-u_\lambda)^{3}}dx
\\ \ \  -\theta \int_{B_{r_n}(0)} \frac{\Delta^s a}{(a-u_\lambda)^{1-2\theta}}dx& +\theta \epsilon_\lambda^{2\theta-1}\int_{B_{r_n}(0)} \Delta^s a dx
  +\theta \lambda\int_{B_{r_n}(0)}\frac{\epsilon_\lambda^{2\theta-1}}{(a-u_\lambda)^2}dx.
\end{align*}

Since the mapping $\lambda\mapsto u_\lambda$ is strictly increasing, we have that
$$\epsilon_\lambda =a(r_n)-u_\lambda(r_n)  \ge a(r_n)-u_{\lambda^*}(r_n):=\varepsilon_n>0$$
and  it infers by $a(x)=\kappa(1-|x|^2)^{\frac23s}$,
$$|\Delta^s a|\le C_{n,s}\quad {\rm in}\quad B_{r_n}(0)$$
 for some $C_{n,s}>0$, then letting $\theta\in(-2,0)$, we have that $2+\theta>0$ and by  H\"{o}lder inequality, we obtain that
\begin{align*}
&\lambda(2+\theta)\int_{B_{r_n}(0)}\frac1{(a-u_\lambda)^{3-2\theta}}dx
\\ \le &  \int_{B_{r_n}(0)} \frac{4\lambda\epsilon_\lambda^\theta}{(a-u_\lambda)^{3-\theta}}dx-\theta\int_{B_{r_n}(0)} \frac{C_{n,s}}{(a-u_\lambda)^{1-2\theta}}  dx
-\theta\varepsilon_n^{2\theta-1} C_{n,s} |B_{r_n}(0)|\nonumber
\\ \le  &  4\lambda^*\varepsilon_n^{\theta} |B_1(0)|^{ \frac{-\theta}{3-2\theta}}\left(\int_{B_{r_n}(0)}\frac{1}{(a-u_\lambda)^{3-2\theta}}dx\right)^{\frac{3-\theta}{3-2\theta}}
\\&  -\theta C_{n,s} |B_1(0)|^{ \frac{2}{3-2\theta}} \left(\int_{B_{r_n}(0)}\frac{1}{(a-u_\lambda)^{3-2\theta}}dx\right)^{\frac{1-2\theta}{3-2\theta}}
-\theta \varepsilon_n^{2\theta-1} C_{n,s} |B_1(0)|,
\end{align*}
thus,
there exists $c_{65}>0$ independent on $ \lambda$ such that
\begin{equation}\label{4.3.4}
 \int_{B_{r_n}(0)}\frac1{(a-u_\lambda)^{3-2\theta}}dx\le c_{65}.
\end{equation}
When $1\le N\le \frac{14s}{3}$, $\frac{3N}{2s}\le 3-2\theta$ for some $\theta\in(-2,0)$, by Lemma \ref{lm 4.2},
we have that $u_{\lambda}$ has uniformly in $C^{2,\beta}_{loc}(B_1(0))$ with $\beta<\gamma$, then $u_{\lambda^*}$
is a classical solution of (\ref{eq 1.1}) with $\lambda=\lambda^*$ and
$a-u_{\lambda^*}>0$ in $B_1(0)$.  \hfill$\Box$

\bigskip

\noindent{\bf Proof of Theorem \ref{teo 3}.}
Proposition \ref{pr 4.1} shows that $u_{\lambda^*}$ is a weak solution of (\ref{eq 1.1}) with $\lambda=\lambda^*$.
The stability of $u_\lambda$   follows by Proposition \ref{pr 4.2} for $\lambda\in(0,\lambda_*)$.
 When $\gamma=\frac23s$, the stability and regularity of $u_{\lambda^*}$ are obtained by Proposition \ref{pr 4.3} and Proposition \ref{pr 4.4}. \hfill$\Box$

\bigskip

\noindent {\bf  Conflicts of interest:} The authors declare that
they have no conflicts of interest regarding this work.

\medskip

\noindent {\bf Data availability:} This paper has no associated
data.

\medskip

\noindent{\small {\bf Acknowledgements:} 
This work is supported by the National Natural Science Foundation of China (Nos.12361043, 12371114),
the Natural Science Foundation of Jiangxi Province  (Nos.20252BAC240158, 20232ACB201001, 20232ACB211001),
and the National Key R$\&$D Program of China (No.2022YFA1005601).}

\end{document}